\documentclass[a4paper, 
headsepline=off, DIV=12, titlepage=false, 11pt]{scrartcl}
\setkomafont{author}{\sffamily\scshape}
\title{\vspace{-1cm}\LARGE The equivariant Tamagawa Number Conjecture\\
for abelian extensions of imaginary quadratic fields}
\date{}
\author{Dominik Bullach \and Martin Hofer}
 
 \usepackage[backend=bibtex
 , style=alphabetic
 , sorting=nyt, maxbibnames=9
 ]{biblatex}
\renewbibmacro{in:}{%
  \ifentrytype{article}{}{\printtext{\bibstring{in}\intitlepunct}}}
  \DeclareFieldFormat{journaltitle}{#1\isdot}
  \DeclareFieldFormat[article, inproceedings, unpublished]{title}{\textit{#1}}

    \usepackage{anyfontsize}
	\usepackage{amsmath}	
	\usepackage{amssymb}
	\usepackage{amsthm}
	\usepackage[british]{babel}
	\usepackage{bm}
	\usepackage{comment}
	\usepackage{datetime}
	\usepackage{enumitem}
	\usepackage[inner=26mm, outer=26mm, 
	top=22mm, 
	bottom=32mm,
				a4paper]{geometry}
	\usepackage[framemethod=TikZ, roundcorner=1mm, linecolor=gray]{mdframed}
	\usepackage{HanFakt} 
	\usepackage{tikz-cd}
	\usepackage{url}
	\usepackage{stmaryrd}
	\usepackage{xcolor}
	\usepackage{mathdots}
	\usepackage{mathrsfs}
	\usepackage{mathtools}
	\usepackage{todonotes}
	
\usepackage{tikz-cd, tikz-3dplot, pgfplots}
\usetikzlibrary{patterns, arrows, positioning, decorations.markings}
\tikzset{axis/.style={very thick, ->}}
\tikzset{
 		  every pin/.style={pin edge={<-}},
  		>=stealth,
   		flow/.style={decoration={markings,mark=at position #1 with {\arrow[black, 
			line width=1.3pt]{>}}},
       		postaction={decorate}},
   		flow/.default=0.5,
		flow_inv/.style={decoration={markings,mark=at position #1 with 
			{\arrow[black, line width=1.3pt]{<}}},
       		postaction={decorate}},
   		flow_inv/.default=0.5,
  		eigendirection/.style={dashed, grau, very thick}
  	}
	
	\tikzset{graph/.style={grau}}
	
\tikzset{
    slanted/.style={rotate=-90, anchor = south},
    slantedswap/.style={rotate=-90, anchor=north}
}

\usepackage{xpatch}
\xapptocmd\normalsize{%
 \abovedisplayskip=5pt plus 3pt minus 1pt
 \abovedisplayshortskip=3pt plus 2pt minus 2pt
 \belowdisplayskip=5pt plus 1pt minus 3pt
 \belowdisplayshortskip=3pt plus 2pt minus 2pt
}{}{}

	\renewmdenv[%
   	 backgroundcolor=gray!25,
   	 linecolor=white,
    	outerlinewidth=1pt,
    	roundcorner=3mm,
	]{boxed}	

    \newtheorem{thmintro}{Theorem}

	\swapnumbers
	\theoremstyle{plain}
	
	\newtheorem{thm}{Theorem}[section]
	\newtheorem{conj}[thm]{Conjecture}
	\newtheorem{prop}[thm]{Proposition}
	\newtheorem{cor}[thm]{Corollary}
	\newtheorem{lem}[thm]{Lemma}
	\newtheorem{definition}[thm]{Definition}

	\newtheorem{conjecture}[thm]{Conjecture}
	
	\theoremstyle{definition}
	
	\newtheorem{bsp1}[thm]{Example}
	\newtheorem{bspe1}[thm]{Examples}
	\newtheorem{rk}[thm]{Remark}

	\newtheorem{condition}[thm]{Condition}

	\newtheorem{remark}[thm]{Remark}

    \swapnumbers

	\newenvironment{bspe}
	{\begin{bspe1}
	\begin{liste}}
	{\end{liste}
	\end{bspe1}}

\setlist[enumerate]{topsep=3pt, itemsep=3pt, parsep=0pt}
	\newenvironment{liste}
	{\begin{enumerate}[label=(\alph*), topsep=3pt]
	\setlength{\itemsep}{3pt}
  \setlength{\parskip}{0pt}
  \setlength{\parsep}{0pt}}
	{\end{enumerate}}

	\newenvironment{cdiagram}
	{\begin{equation*} \begin{tikzcd}}
	{\end{tikzcd} \end{equation*}}

    \newcommand{\anti}{\mathrm{anti}}
    \newcommand{\Col}{\mathrm{Col}}
	\renewcommand{\:}{\colon}
	\renewcommand{\a}{\mathfrak{a}}

	\newcommand{\bigO}{\mathcal{O}}

	\newcommand{\C}{\mathbb{C}}
	\renewcommand{\emptyset}{\varnothing}
	\renewcommand{\emph}[1]{\textit{\textbf{#1}}}
	
	\newcommand{\ff}{\mathfrak{f}}
	\newcommand{\Fitt}{\operatorname{Fitt}}
	\newcommand{\Frob}{\mathrm{Frob}}
	\newcommand{\gal}[2]{\mathrm{Gal}( #1 / #2)}
	\renewcommand{\iff}{\quad \Leftrightarrow \quad}
	\newcommand{\id}{\mathrm{id}}

	 \newcommand{\N}{\mathbb{N}}
	\newcommand{\nZ}[1]{\faktor{\mathbb{Z}}{#1 \mathbb{Z}}}
	\newcommand{\p}{\mathfrak{p}}
	
	\newcommand{\m}{\mathfrak{m}}
	\newcommand{\q}{\mathfrak{q}}
	\newcommand{\Q}{\mathbb{Q}}
	\newcommand{\Ord}{\mathrm{Ord}}
	\newcommand{\ord}{\mathrm{ord}}
	\newcommand{\R}{\mathbb{R}}

	\newcommand{\Rec}{\mathrm{Rec}}
	\newcommand{\rec}{\mathrm{rec}}

	\renewcommand{\ord}{\operatorname{ord}}

	\newcommand{\UN}{\mathrm{UN}}
	\newcommand{\Z}{\mathbb{Z}}

\usepackage[colorlinks,
pdfpagelabels,
pdfstartview = FitH,
bookmarksopen = true,
bookmarksnumbered = true,
linkcolor = black,
plainpages = false,
hypertexnames = false,
citecolor = black, urlcolor=black]{hyperref}

\newcommand{\nocontentsline}[3]{}
\newcommand{\tocless}[2]{\bgroup\let\addcontentsline=\nocontentsline#1{#2}\egroup}

  1

\DeclareMathOperator{\Det}{Det}

\DeclareMathOperator{\Hom}{Hom}

\DeclareMathOperator{\Spec}{Spec}

\DeclareMathOperator{\im}{im}

\newcommand{\NN}{\mathrm{N}}

\newcommand{\cQ}{\mathcal{Q}}

\newcommand{\cG}{\mathcal{G}}
\newcommand{\cK}{\mathcal{K}}

\newcommand{\frp}{\mathfrak{p}}

\def\bigcapp{\raise1ex\hbox{\rotatebox{180}{$\biguplus$}}}
 \def\bigcappd{\raise1ex\hbox{\rotatebox{180}{$\displaystyle\biguplus$}}}

 \newcommand{\cyc}{\mathrm{cyc}}
 
 \newcommand{\tor}{\mathrm{tor}}
 \newcommand{\tf}{\mathrm{tf}}

 \newcommand{\ram}{\mathrm{ram}}
 
 \newcommand{\bidual}{\bigcap\nolimits}
\newcommand{\exprod}{\bigwedge\nolimits}

\newcommand{\spc}{\mathrm{split}}

\newcommand{\Iw}{\mathrm{Iw}}

  \tikzset{
    rotated/.style={rotate=-90, anchor = south},
    rotatedswap/.style={rotate=-90, anchor=north, outer sep=0.75mm}
}

\newcommand{\bLambda}{{\mathpalette\makebLambda\relax}}
\newcommand{\makebLambda}[2]{%
  \raisebox{\depth}{\scalebox{1}[-1]{$\mathsurround=0pt#1\mathbb{V}$}}%
}

\usepackage{dsfont}
\renewcommand{\mathbb}{\mathds}

	\tikzcdset{
  cells={font=\everymath\expandafter{\the\everymath\displaystyle}},
}

	\usepackage{tocloft}

\begin{document}
\maketitle

\vspace{-1.5cm}
\begin{abstract}
We prove the Iwasawa-theoretic version of a Conjecture of Mazur--Rubin and Sano in the case of elliptic units. This allows us to derive the $p$-part of the equivariant Tamagawa number conjecture at $s = 0$ for abelian extensions of imaginary quadratic fields in the semi-simple case and, provided that a standard $\mu$-vanishing hypothesis is satisfied, also in the general case. 
\end{abstract}


\let\thefootnote\relax\footnotetext{
\hskip 0.155truein 
}


\section{Introduction}

The \textit{equivariant Tamagawa Number Conjecture} (eTNC for short) as formulated by Burns and Flach \cite{BurnsFlach01}  (building on earlier work of Kato \cite{Kato93a}, \cite{Kato93b} and, independently, Fontaine and Perrin-Riou \cite{FontainePerrinRiou94}) is an equivariant refinement of the seminal Tamagawa Number Conjecture of Bloch and Kato \cite{bloch-kato}. It both unifies and refines a great variety of conjectures related to special values of motivic $L$-functions such as Stark's conjectures, the Birch and Swinnerton-Dyer conjecture, and the central conjectures of classical Galois module theory (see \cite{Burns07}, \cite{GuidoKings}, \cite{Burns01} for more details). \\
The idea of deducing cases of the eTNC from a variant of the Iwasawa Main Conjecture already appears in the original article of Bloch and Kato \cite{bloch-kato}. However, the necessary descent calculations are particularly involved in cases where the associated $p$-adic $L$-function posseses so-called \textit{trivial zeroes}. To handle such cases, Burns and Greither \cite{BurnsGreither} developed a descent machinery in their proof of the eTNC for the Tate motive $(h^0 (\Spec K), \Z [\frac12] [\gal{K}{\Q}])$, where $K$ denotes an absolutely abelian field (the $2$-part was later resolved by Flach \cite{Flach11}). This formalism uses the vanishing of certain Iwasawa $\mu$-invariants, the known validity of the Gross--Kuz'min conjecture in this setting, and a result of Solomon \cite{Solomon1992} as crucial ingredients. Bley \cite{Ble06} later proved partial results for $K$ an abelian extension of an imaginary quadratic field using the same strategy and an analogue \cite{Bley04} for elliptic units of Solomon's result. \\
In \cite{BKS2} Burns, Kurihara and Sano showed that an Iwasawa-theoretic version of a conjecture proposed by Mazur--Rubin \cite{MazurRubin} and, independently, Sano \cite{Sano} constitutes an appropriate conjectural generalisation of the aforementioned result of Solomon's and therefore allows one to extend the Burns--Greither descent formalism to provide a general strategy for proving $\mathrm{eTNC} (h^0 (\Spec K), \Z [\gal{K}{k}])$, where $K / k$ is a finite abelian extension of number fields. \smallskip \\
In the present article, we prove the following result (see Theorem \ref{MRS-iq-thm} for the precise statement).

\begin{thmintro} \label{thm-A}
The Iwasawa-theoretic Mazur--Rubin--Sano Conjecture~\ref{MRS} holds for elliptic units.
\end{thmintro}

It is perhaps worth noting that Theorem \ref{thm-A} also includes the often technically difficult cases $p = 2$ and $p$ being non-split in the imaginary quadratic base field. This was not possible only using techniques previously employed for the absolute abelian situation but is achieved by building on new results from \cite{BlHo20}.\\
We will then use Theorem \ref{thm-A} and the formalism of Burns, Kurihara and Sano \cite{BKS2} to deduce new cases of the $p$-part of the eTNC from the relevant equivariant Iwasawa Main Conjecture.

\begin{thmintro} \label{thm-B}
Let $p$ be a prime number, $k$ an imaginary quadratic field, and $K / k$ a finite abelian extension. 
\begin{liste}
\item If $p$ is split in $k$, then $\mathrm{eTNC} ( h^0 (\Spec K), \Z_p [\gal{K}{k}])$ holds.
\item If $p$ is not split in $k$, then $\mathrm{eTNC} ( h^0 (\Spec K), \Z_p [\gal{K}{k}])$ 
holds if $p \nmid [K : k]$ or the classical Iwasawa $\mu$-invariant vanishes (see Theorem \ref{thm_iq_main_thm} for a more precise statement).
\end{liste}
In particular, 
$\mathrm{eTNC}(h^0 (\Spec (K)), \Z[\gal{K}{k}])$ holds if $[K : k]$ is a prime power or every prime factor of $[K : k]$ is split in $k$ (see Corollary \ref{cor_finsterau}).
\end{thmintro}
The first part of Theorem \ref{thm-B} generalises work of Bley \cite{Ble06} which only covers prime numbers $p> 2$ that do not divide the class number of $k$.
The second part of Theorem \ref{thm-B} grew out of the second presently named author's thesis \cite{hofer} and not only settles the descent problem in this previously widely open case but also provides for a large supply of new examples in which the eTNC is valid unconditionally.
More precisely, the condition on the class number imposed in \cite{Ble06} meant that the unconditional validity of the eTNC was previously only known in certain cases where $k$ is one of only nine imaginary quadratic fields of class number one. In contrast, Theorem \ref{thm-B} is devoid of any such restrictive hypotheses on the field $k$.
 In this regard, Theorem \ref{thm-B} also improves on a recent result of Burns, Daoud, Seo, and the first author \cite{scarcity} which does not establish the unconditional validity of eTNC in any cases, in fact not even of certain $p$-parts thereof. 
\smallskip \\
The proof of Theorem \ref{thm-B}\,(b) also requires the validity of 
an appropriate analogue of the Gross--Kuz'min conjecture which is labelled condition `(F)' in \cite{BKS2}. 
 To this end, we prove the following general result  which
 seems to not have previously appeared in the literature (see Theorems \ref{new-gross-kuzmin-thm-cyclotomic} and \ref{exists-Zp-extension} for the full statements). 
 We remark that Theorem \ref{thm-C}~(a) is a generalisation 
 of Gross's classical result on the minus part of the aforementioned conjecture in the setting of CM extensions of totally real fields (cf.\@ Remark \ref{gross-kuzmin-rk-1}). 
 
\begin{thmintro}\label{thm-C}
Let $K / k$ be an abelian extension of number fields and let $p$ be a prime number. If $p = 2$, assume that $k$ is totally imaginary. 
\begin{liste}
    \item Let $\chi$ be a non-trivial character on $\gal{K}{k}$ and
    let $k_\infty / k$ be a $\Z_p$-extension in which no finite place splits completely.\\
    If $p$ splits completely in $k / \Q$ and there is at most one finite place $v$ of $k$ that both ramifies in $k_\infty / k$ and is such that $\chi (v) = 1$, then the
    validity of condition (F) for the $\Z_p$-extension $K \cdot k_\infty$ of $K$ is implied by the validity of condition (F) for the $\Z_p$-extension $k_\infty$ of $k$.
    \item If $k$ is imaginary quadratic and $p$ does not split in $k / \Q$, then there are infinitely many $\Z_p$-extensions $k_\infty$ of $k$ such that condition (F) holds for the $\Z_p$-extension $K \cdot k_\infty$ of $K$
\end{liste}
\end{thmintro}

The main contents of this article are as follows. In \S\,\ref{Rubin--Stark-section} we recall the definition of Rubin--Stark elements and the Rubin--Stark Conjecture. In \S\,\ref{set-up-section} we introduce the objects and notations in Iwasawa theory that will be used throughout most of the article, and collect useful preliminary results on Iwasawa cohomology complexes and universal norms. In \S\,\ref{congruences-section} we recall the Iwasawa-theoretic Mazur--Rubin--Sano Conjecture and give a more explicit reformulation of the conjecture. In \S\,\ref{big-gross-kuzmin-section} we study the Gross--Kuz'min Conjecture and condition (F), and prove Theorem \ref{thm-C}. In \S\,\ref{iq-section} we finally specialise to imaginary quadratic base fields and prove Theorems~\ref{thm-A} and \ref{thm-B}.
\smallskip \\
\textbf{\sffamily Acknowledgements}\,
The authors would like to extend their gratitude to Werner Bley, David Burns, Alexandre Daoud, S\"oren Kleine, Takamichi Sano and Pascal Stucky for many illuminating conversations and helpful comments on earlier versions of this manuscript.\\
The first author wishes to acknowledge the financial support of the Engineering  and  Physical  Sciences  Research  Council [EP/L015234/1],  the  EPSRC  Centre  for  Doctoral  Training  in  Geometry  and  Number  Theory  (The  London School of Geometry and Number Theory), King's College London and University College London.\\
An earlier version of this manuscript was circulated under the title `On trivial zeroes of Euler systems for $\mathbb{G}_m$', various ideas of which have subsequently been moved to an article of Burns, Daoud, Seo and the first author \cite{scarcity}.  
\smallskip \\
\textbf{\sffamily Notation}\, \textit{Arithmetic.} 
For any number field $E$ we write $S_\infty(E)$ for the set of archimedean places of $E$, and $S_p(E)$ for the set of places of $E$ lying above a rational prime $p$.
Given an extension $F/ E$ we write $S_\ram(F/E)$ for the places of $E$ that ramify in $F$. 
If $S$ is a set of places of $E$, we denote by $S_F$ the set of places of $F$ that lie above those contained in $S$. We will however omit the explicit reference to the field in case it is clear from the context. For example, $\bigO_{F, S}$ shall denote the ring of $S_F$-integers of $F$, and $U_{F, S} = \Z_p \otimes_\Z \bigO_{F, S}^\times$ the $p$-completion of its units. We also define $Y_{F, S}$ to be the free abelian group on $S_F$ and set
\[
X_{F, S} = \Big \{ \sum_{w \in S_F} a_w w \in Y_{F, S} \mid \sum_{w \in S_F} a_w = 0 \Big \}.
\]
Furthermore, if $T$ is a finite set of finite places disjoint from $S$, then we let $A_{S, T} (F)$ be the $p$-part of the $S_F$-ray class group mod $T_F$, i.e.\@ the $p$-Sylow subgroup of the quotient of the group of fractional ideals of $\bigO_{F, S}$ coprime to $T_F$ by the subgroup of principal ideals with a generator congruent to 1 modulo all $w \in T_F$. If
$S = S_\infty (E)$ or $T = \emptyset$, then we will suppress the respective set in the notation. \\
For any place $w$ of $F$ we write $\ord_w \: F^\times \to \Z$ for the normalised valuation at $w$. In case of a finite extension $H$ of $\Q_p$ we also write $\ord_H$ for the normalised valuation on $H$. If $F / E$ is abelian and $v$ unramified in $F / E$, then we let $\Frob_v \in \gal{F}{E}$ be the arithmetic Frobenius at $v$. If $v$ is a finite place of $E$, then we denote by $\NN v = | \faktor{\bigO_E}{v}|$ the norm of $v$.
\medskip \\
\textit{Algebra.} For an abelian group $A$ we denote by $A_\tor$ its torsion-subgroup and by $A_\tf = \faktor{A}{A_\tor}$ its torsion-free part. If there is no confusion possible, we often 
shorten the functor $(-) \otimes_\Z A$ to just $(-) \cdot A$ (or even $(-) A$) and, if $A$ is also a $\Z_p$-module, similarly for the functor $(-) \otimes_{\Z_p} A$. 
If $A$ is finite, we denote by $\widehat{A} = \Hom_\Z (A, \C^\times)$ its character group, and for any $\chi \in \widehat{A}$ we let
\[
e_\chi = \frac{1}{|A|} \sum_{\sigma \in A} \chi (\sigma) \sigma^{-1} \quad \in \C [A]
\]
be the usual primitive orthogonal idempotent associated to $\chi$. Furthermore, $\NN_A = \sum_{\sigma \in A} \sigma \in \Z [A]$ denotes the norm element of $A$. 
\medskip \\
If $R$ is a commutative Noetherian ring, then for any $R$-module $M$ we write $M^\ast = \Hom_R (M, R)$ for its dual and $\Fitt^0_R (M)$ for its (initial) Fitting ideal. Let $r \geq 0$ be an integer, then the $r$-th \textit{exterior bidual} of $M$ is defined as
\[
\bidual^r_R M = \left ( \exprod^r_R M^\ast \right )^\ast.
\]
If $R = \Z [A]$ for a finite abelian group $A$, then the exterior bidual coincides with the lattice first introduced by Rubin in \cite[\S\,2]{Rub96}, see \cite[Rk.~A.9]{EulerSystemsSagaI} for the relation between these two definitions. The theory of exterior biduals has since seen great development and the reader is invited to consult, for example, \cite[App.~A]{EulerSystemsSagaI} or \cite[\S\,2]{BullachDaoud} for an overview. At this point we only remark that for $r \geq 1$ any $f \in  M^\ast$ induces a map 
\[
\bidual^r_R M \to \bidual^{r - 1}_R M
\]
which, by abuse of notation, will also be denoted by $f$, and is defined as the dual of
\[
\exprod^{r - 1}_R M^\ast \to \exprod^r_R M^\ast, \quad g \mapsto f \wedge g.
\]
Iterating this construction gives, for any $s \leq r$, a homomorphism
\begin{equation} \label{biduals-duals-hom}
\exprod^s_R M^\ast \to \Hom_R \big ( \bidual^r_R M, \; \bidual^{r - s}_R M\big), \quad
f_1 \wedge \dots \wedge f_s \mapsto f_s \circ \dots \circ f_1.
\end{equation}
Finally, we write $\mathcal{Q} (R)$ for the total ring of fractions of $R$. 

\section{Rubin--Stark elements} \label{Rubin--Stark-section}

Let $K / k$ be a finite abelian extension of number fields with Galois group $\cG \coloneqq \gal{K}{k}$ and fix a finite set $S$ of places of $k$ which contains $S_\infty (k) \cup S_\ram (K / k)$. Suppose to be given a propert subset $V \subseteq S$ of places which split completely in $K / k$ and choose an ordering $S = \{ v_0, \dots, v_t \}$ such that $V = \{ v_1, \dots, v_r \}$. For every $i \in \{0, \dots, t \}$ fix a place $\overline{v_i}$ of the algebraic closure $\overline{\Q}$ of $\Q$ that extends $v_i$ and write $w_i = w_{K, i}$ for the place of $K$ induced by $\overline{v_i}$. \smallskip \\
The \textit{Dirichlet regulator} map 
\begin{equation} \label{dirichlet-regulator}
\lambda_{K, S} \: \bigO_{K, S}^\times \to \R X_{K, S}, 
\quad a \mapsto - \sum_{w \in S_K} \log | a |_w \cdot w 
\end{equation}
then induces
an isomorphism
\begin{equation} \label{dirichlet-isomorphism}
\R \exprod^r_{\Z [\cG]} \bigO_{K, S}^\times \stackrel{\simeq}{\longrightarrow} \R \exprod^r_{\Z [\cG]}  X_{K, S}
\end{equation}
that will also be denoted as $\lambda_{K, S}$.
If $\chi \in \widehat{\cG}$ and $T$ is a finite set of places of $k$ disjoint from $S$, we moreover define the $S$-truncated $T$-modified \textit{Artin $L$-function} as 
\[
L_{K / k, S, T} (\chi, s) = \prod_{v \in T} (1 - \chi (\Frob_v) \NN v^{1 - s}) \cdot \prod_{v \not \in S} (1 - \chi (\Frob_v) \NN v^{-s})^{-1},
\]
where $s$ is a complex number of real part strictly greater than 1. It is well-known that this defines a function on the complex plane by meromorphic continuation. By \cite[Ch.\@ I, Prop.\@ 3.4]{Tate}, the existence of the set $V \subsetneq S$ then implies that the order of vanishing of $L_{K / k, S, T} (\chi, s)$ at $s = 0$ is at least $r$. This allows us to define the $r$-th order \textit{Stickelberger element} as 
\[
\theta^{(r)}_{K / k, S, T} (0) = \sum_{\chi \in \widehat{\cG}} e_\chi \cdot \lim_{s \to 0} s^{-r} L_{K / k, S, T} (\overline{\chi}, s) \quad
\in \R [\cG].
\]

\begin{definition}
The $r$-th order \textit{Rubin--Stark element} $\varepsilon^V_{K / k, S, T} \in \R \exprod^r_{\Z [\cG]} \bigO_{K, S}^\times$ is defined to be the preimage of the element $\theta^{(r)}_{K / k, S, T} (0) \cdot \bigwedge_{1 \leq i \leq r} ( w_{i} - w_{0})$ under the isomorphism (\ref{dirichlet-isomorphism}) induced by the Dirichlet regulator map $\lambda_{K, S}$. 
\end{definition}
To  state the $p$-part of the Rubin--Stark Conjecture for a prime number $p$ we now fix an isomorphism $\C \cong \C_p$ that allows us to regard $\varepsilon^V_{K / k, S, T}$ as an element of $\C_p \exprod^r_{\Z_p [\cG]} U_{K, S}$. We also write $U_{K, S, T} \coloneqq \Z_p \otimes_\Z \bigO^\times_{K, S, T}$ for the $p$-completion of the group of $(S_K, T_K)$-units which are defined as 
\[
\bigO_{K, S, T}^\times = \ker \Big \{ \bigO_{K, S}^\times \to \bigoplus_{w \in T_K} \big ( \faktor{\bigO_{K}}{w} \big )^\times \Big \}
\]
and will often assume that $T$ is chosen in a way such that $U_{K, S, T}$ is $\Z_p$-torsion free (which is automatically satisfied if $T$ is non-empty) 

\begin{conj} \label{Rubin--Stark-conj}
If $U_{K, S, T}$ is $\Z_p$-torsion free, then $\varepsilon^V_{K / k, S, T}$ belongs to $\bidual^r_{\Z_p [\cG]} U_{K, S, T}$. 
\end{conj}

\begin{bspe} \label{Rubin--Stark-examples}
\item (\textit{cyclotomic units}) Take $k = \Q$, $S = S_\infty (\Q) \cup S_\ram (K / k)$, $K$ a finite real abelian extension of $\Q$, and $V = S_\infty (\Q) = \{ v_1 \}$. 
If we set $\delta_T \coloneqq \prod_{v \in T} (1 - \NN v \Frob_v^{-1})$, then one has
\[
\varepsilon^V_{K / \Q, S, T} = \delta_T \cdot \big( \frac12 \otimes N_{\Q (\xi_m) / K} (1 - \xi_m) \big)
\quad \in \bigO^\times_{K, S, T},
\]
where $m = m_K$ is the conductor of $K$ and $\xi_m = \iota^{-1} ( e^{2 \pi i / m})$ for the embedding $\iota \: \overline{\Q} \hookrightarrow \C$ corresponding to the choice of place $\overline{v_1}$ fixed at the beginning of the section (see \cite[Ch.~IV, \S\,5]{Tate} for a proof). In particular, the $p$-adic Rubin--Stark Conjecture \ref{Rubin--Stark-conj} holds for all primes $p$.

\item (\textit{Stickelberger elements}) Let $k$ be a totally real field, $S = S_\infty (k) \cup S_\ram (K / k)$, $K$ a finite abelian CM extension of $k$, and $V = \emptyset$. In this setting Conjecture \ref{Rubin--Stark-conj} hold true for all primes $p$ due to the work of Deligne--Ribet \cite{DeligneRibet} and the Rubin--Stark element is given by
\[
\varepsilon^V_{K / k, S, T} = \theta_{K / k, S, T} (0). 
\]
\item (\textit{elliptic units}) Let $k$ be an imaginary quadratic field and $\mathfrak{f} \subseteq \bigO_k$ a non-zero ideal such that $\bigO_k^\times \to ( \faktor{\bigO_k}{\mathfrak{f}})^\times$ is injective. Take $S = S_\infty (k) \cup S_\ram (K / k) \cup \{ \q \mid \mathfrak{f}\}$, $K$ a finite abelian extension of $k$, and $V = S_\infty (k) = \{ v_1 \}$. Then the Rubin--Stark Conjecture holds for all fields $E \in \Omega ( \cK / k)$ satisfying $|S| > 1$ (which we may assume always to be true after possibly enlarging $\mathfrak{f}$), see, for example, \cite[Ch.~IV, Prop.~3.9]{Tate}. To describe the Rubin--Stark element in this setting, we write $\m = \m_K$ for the conductor of $K$. Let $k( \mathfrak{f} \m)$ be the ray class field of $k$ modulo $\mathfrak{f} \m$ and choose an auxiliary prime ideal $\a \subsetneq \bigO_k$ coprime to $6 \mathfrak{f} \m$. Using the elliptic function $\psi$ introduced by Robert \cite{Rob92} we set
\[
\psi_{\ff \m, \a} = \iota^{-1} (\psi (1; \ff \m, \a)) \quad \in \bigO_{k (\ff \m), S}^\times
\]
for the embedding $\iota \: \overline{\Q} \hookrightarrow \C$ corresponding to $\overline{v_1}$, where $\psi(1; \ff \m, \a)$ is a common short hand for what would be $\psi(1; \ff \m, \a^{-1} \ff \m )$ in Robert's original notation.
It then follows from Kronecker's second limit formula, e.g.\@  \cite[Lem.~2.2\,e)]{Fla09}, that 
\[
\varepsilon^V_{K / k, S, \{ \a \}} = \NN_{k (\ff \m) / E} ( \psi (1; \ff \m, \a))
\in \bigO_{E,S,\{ \a \}}^\times.
\]
\end{bspe}

\section{Preparations in Iwasawa theory} \label{set-up-section}

Throughout this section we fix a number field $K$ and a rational prime $p$. We suppose to be given a finite abelian extension $K / k$ and a $\Z_p$-extension $k_\infty / k$ in which all infinite places split completely (this is automatic if $p$ is odd). We define $K_\infty \coloneqq K k_\infty$ and write $K_n$ for the $n$-th layer of the $\Z_p$-extension $K_\infty / K$ (here $K_0$ means $K$). 
We further set $\Gamma_n \coloneqq \gal{K_n}{K}$, $\Gamma^n \coloneqq \gal{K_\infty}{K_n}$, $\cG_n \coloneqq \gal{K}{k}$, and $\bLambda \coloneqq \Z_p \llbracket \gal{K_\infty}{k} \rrbracket$. In case $n = 0$ we suppress any reference to $n$ in the notation. \smallskip \\
Moreover, we introduce the following notations and assumptions:
\begin{enumerate}[label=$\bullet$]
    \item $S$ a finite set of places of $k$ which contains $S_\infty (k) \cup S_\ram (K / k)$ and is such that not finite place in $S$ splits completely in $k_\infty /k$,
    \item $\Sigma \coloneqq S \cup S_\ram (k_\infty /k)$,
    \item $V\subsetneq \Sigma$ the subset of places which split completely in $K_\infty / k$ (by our assumptions this is a subset of $S_\infty (k)$),  and $r$ its cardinality, 
    \item $V' \subsetneq \Sigma$ a set of places which contains $V$ and consists of places that split completely in $K / k$ and has cardinality $r'$,
    \item $T$ a finite set of places of $k$ which is disjoint from $\Sigma$, contains only places that do not split completely in $k_\infty / k$, and is such that $U_{K_n, \Sigma, T}$ is $\Z_p$-torsion free for all integers $n \geq 0$. 
\end{enumerate}
We fix a labelling $\Sigma = \{ v_0, \dots, v_t\}$ such that $V = \{ v_1, \dots, v_r \}$ and $V' = \{ v_1, \dots, v_{r'}\}$.

\begin{lem} \label{silly-little-lemma}
For any topological generator $\gamma \in \Gamma$ the element $\gamma - 1 \in \bLambda$ is a non-zero divisor.
\end{lem}

\begin{proof}
Fix  a splitting $\gal{K_\infty}{k} \cong \Gamma' \times \Delta$ with $\Gamma' \cong \Z_p$ and $\Delta$ finite. Set $L \coloneqq K_\infty^{\Gamma'}$ and write $L_n$ for the $n$-th layer of the $\Z_p$-extension $K_\infty / L$. 
By definition we have $K_\infty = \bigcup_{n \geq 0} L_n$, hence there is $n$ such that $K \subseteq L_n$. That is, $L_n$ is an intermediate field of the $\Z_p$-extension $K_\infty / K$ and therefore must agree with $K_m$ for some $m$. To prove the Lemma we fix a topological generator $\gamma_L \in \Gamma'$. Then there is a unit $a \in \Z_p^\times$ such that $\gamma_L^{a p^n} = \gamma^{p^m}$. 
The element $\gamma_L^{a p^n} - 1$ is clearly a non-zero divisor in $\bLambda = \Z_p \llbracket \Gamma' \rrbracket [\Delta]$. It then follows from
\[
\gamma_L^{a p^n} - 1 = \gamma^{p^m} - 1 = (\gamma - 1) \cdot (1 + \gamma + \dots + \gamma^{p^m - 1})
\]
that $\gamma - 1$ must be a non-zero divisor as well.
\end{proof}

\subsection{Modified Iwasawa cohomology complexes}

For any finite abelian extension $E / k$ and finite set of places $M \supseteq S_\infty \cup S_\ram ( E / k)$ that is disjoint from a second finite set of places $Z$, Burns-Kurihara-Sano have constructed \cite[Prop.~2.4]{BKS} a canonical $Z$-modified, compactly supported \textit{Weil-\'etale cohomology complex} $\text{R} \Gamma_{c, Z} ( ( \bigO_{E, M})_\mathcal{W}, \Z)$ of the constant sheaf $\Z$ on the \'etale site of $\Spec \bigO_{E, M}$. 
In the sequel we shall need the complex
\[
C^\bullet_{E, M, Z} = \text{R} \Hom_\Z ( \text{R} \Gamma_{c, Z} ( ( \bigO_{E, M})_\mathcal{W}, \Z), \Z) [-2]
\]
as well as its $p$-completion $D^\bullet_{E, M, Z} \coloneqq \Z_p \otimes_\Z^\mathbb{L} C^\bullet_{E, M, Z}$.
The essential properties of these complexes are listed in \cite[Prop.\@ 3.1]{bdss}. In particular,
there is a canonical isomorphism $H^0 (D^\bullet_{E, M, Z} ) \cong U_{E, M, Z}$ and an exact sequence
\begin{equation} \label{yoneda-finite}
\begin{tikzcd}
0 \arrow{r} & A_{M, Z} (E) \arrow{r} & H^1 (D^\bullet_{E, M, Z}) \arrow{r}{\pi_E} & X_{M, U} \arrow{r} & 0 .
\end{tikzcd}
\end{equation}
Furthermore, there are natural maps $D_{K_{n + 1}, \Sigma, T}^\bullet \to D_{K_n, \Sigma, T}$ in the derived category $D (\Z_p [\cG_{n + 1}])$ of $\Z_p [\cG_{n + 1}]$-modules, which allow us to define the complex \begin{equation} \label{limit-complex}
D^\bullet_{K_\infty, \Sigma, T} = \text{R} \varprojlim_n D^\bullet_{K_n, \Sigma, T}.
\end{equation}

\begin{prop} \label{propeties-iwasawa-complex}
The following claims are valid. 
\begin{liste}
\item The complex $D^\bullet_{K_\infty, \Sigma, T}$ is perfect as an element of the derived category $D (\bLambda)$ and acyclic outside degrees zero and one. 
\item There is a canonical isomorphism $H^0 (D^\bullet_{K_\infty, \Sigma, T} ) \cong U_{K_\infty, \Sigma, T}$ and an exact sequence
\begin{equation} \label{yoneda-Iwasawa}
\begin{tikzcd}
0 \arrow{r} & A_{\Sigma, T} (K_\infty) \arrow{r} & H^1 (D^\bullet_{K_\infty, \Sigma, T}) \arrow{r}{\pi} & X_{K_\infty, \Sigma} \arrow{r} & 0,
\end{tikzcd}
\end{equation}
where $A_{\Sigma, T} (K_\infty) \coloneqq \varprojlim_n A_{\Sigma, T} (K_n)$ and $X_{K_\infty, \Sigma} \coloneqq \varprojlim_n X_{K_n, \Sigma}$.
\item There exists a finitely generated free $\bLambda$-module $\Pi_\infty$, a basis $\{ b_1, \dots, b_d \}$ of $\Pi_\infty$, and an endomorphism $\phi \: \Pi_\infty \to \Pi_\infty$ with the following properties:
\begin{enumerate}[label=(\roman*)]
    \item The complex $\Pi_\infty \stackrel{\phi}{\to} \Pi_\infty$ represents the class of $D^\bullet_{K_\infty, \Sigma, T}$ in $D(\bLambda)$. 
    \item If we set $\Pi_n \coloneqq \Pi_\infty \otimes_\bLambda \Z_p [\cG_n]$ and write $\phi_n$ for the endomorphism of $\Pi_n$ induced by $\phi$, then the complex $\Pi_n \stackrel{\phi_n}{\to} \Pi_n$ represents the class of $D^\bullet_{K_n, \Sigma, T}$ in $D(\Z_p [\cG_n])$. 
    \item If we fix an ordering $\Sigma = \{ v_0, \dots, v_t \}$, then, for any $i \in \{1, \dots, t \}$, the composite map $\Pi_\infty \to H^1 (D^\bullet_{K_\infty, \Sigma, T}) \stackrel{\pi}{\to} X_{K_\infty, \Sigma}$ sends $b_i$ to $( w_{K_n,i} - w_{K_n, 0})_{n \geq 0}$. \label{surjection-pi-maps-places}
\end{enumerate}
\end{liste}
\end{prop}

\begin{proof}
Each complex $D^\bullet_{K_n, \Sigma, T}$ is a complex of compact Hausdorff spaces, hence the inverse limit functor commutes with taking cohomology and so claim (b), and the second part of claim (a), follow by taking the limit of the respective statements for the complexes $D^\bullet_{K_n, \Sigma, T}$. \\
Choose a surjection
\[
\mathrm{pr}_\infty \: \Pi_\infty \twoheadrightarrow H^1 (D^\bullet_{K_\infty, \Sigma, T} ),
\]
where $\Pi_\infty$ is a 
finitely generated free $\bLambda_K$-module of large enough rank $d$ with basis $\{b_1, \dots, b_d \}$, such that condition (iii) of claim (c) is satisfied. 
If we write $\mathrm{pr}_n \: \Pi_n \to H^1 (D^\bullet_{K_n, \Sigma, T})$ for the surjection induced vy $\mathrm{pr}_\infty$, then the method of \cite[\S\,5.4]{BKS} (see also \cite[Prop.\@ A.11\,(i)]{EulerSystemsSagaI}) allows us to choose a representative of $D^\bullet_{K_n, \Sigma, T}$ that is of the form $[Q_n \to \Pi_n]$ with $Q_n$ a $\Z_p [\cG_n]$-projective (hence free) module. By a standard argument in representation theory, the Dirichlet regulator (\ref{dirichlet-isomorphism}) induces a \textit{rational} isomorphism $\Q_p H^0 (D^\bullet_{K_n, \Sigma, T}) \cong \Q_p H^1 (D_{K_n, \Sigma, T})$, therefore we may identify $Q_n \cong \Pi_n$ by Swan's theorem \cite[Thm.\@ (32.1)]{CurtisReiner}.\\
We shall now give an explicit description of the transition maps used in (\ref{limit-complex}) in terms of these fixed representatives $[\Pi_n \stackrel{\phi_n}{\longrightarrow} \Pi_n]$. 
This will give rise to a representative of $D^\bullet_{K_\infty, \Sigma, T}$ that is key to our study. For this purpose, we set $D_n^\bullet \coloneqq D^\bullet_{K_n, \Sigma, T}$ for simplicity and write $\gamma_n$ for the isomorphism $D^\bullet_{n + 1} \otimes^\mathbb{L}_{\Z_p [\cG_{n + 1}]} \Z_p [\cG_{n}] \cong D^\bullet_n$ in the derived category $D (\Z_p [\cG_{n}])$. As a morphism between perfect complexes, this map can be represented by a commutative diagram of the form
\begin{equation} \label{diagram}
\begin{tikzcd}[column sep=small]
0 \arrow{r} & H^0 (D^\bullet_{n + 1} \otimes^\mathbb{L} \Z_p [\cG_{n}]) \arrow{r} \arrow{d}[right]{H^0 (\gamma_n)}[left]{\simeq} & \Pi_n \arrow{r}{\overline{\phi_{n + 1}}} \arrow{d}{\gamma_n^0} & \Pi_n \arrow{d}{\gamma_n^1}
\arrow{r}{\overline{\mathrm{pr}_{n + 1}}} & H^1 (D^\bullet_{n + 1} \otimes^\mathbb{L} \Z_p [\cG_{n}]) \arrow{d}[right]{H^1 (\gamma_n)}[left]{\simeq} \arrow{r} & 0 \\
0 \arrow{r} & H^0 (D^\bullet_{n}) \arrow{r} & \Pi_n \arrow{r}{\phi_n} & \Pi_n 
\arrow{r}{\mathrm{pr}_n} & H^1 (D^\bullet_n) \arrow{r} & 0.
\end{tikzcd}
\end{equation}
Here $\overline{\phi_{n + 1}}$ and $\overline{\mathrm{pr}_{n + 1}}$ denote the maps induced by $\phi_{n + 1}$ and $\mathrm{pr}_{n + 1}$, respectively. By construction, $\mathrm{pr}_n = H^1 (\gamma_n) \circ \overline{\mathrm{pr}_{n + 1}}$ and so exactness of the bottom line in (\ref{diagram}) yields that the image of $\gamma^1_n - \id_{\Pi_n}$ is contained in the image of $\phi_n$. Choose a map $h \: \Pi_n \to \Pi_n$ such that $\phi_n \circ h = \gamma^1_n - \id_{\Pi_n}$ and set $f = \gamma_n^0 - h \circ \overline{\phi_{n + 1}}$, then $h$ defines a chain homotopy between $(\gamma^0_n, \gamma^1_n)$ and $(f, \id)$. We may therefore assume that $\gamma^1_n$ is the identity map. Given this, we can appeal to the Five Lemma to deduce from (\ref{diagram}) that $\gamma^0_n$ is an isomorphism as well. Finally, we may now pass to the limit to obtain a representative $[\Pi_\infty \stackrel{\phi}{\longrightarrow} \Pi_\infty]$ of the complex $D^\bullet_{K_\infty, \Sigma, T}$, as required to prove part (i) of claim (c). In particular, the latter complex is perfect as an element of the derived category $D (\bLambda)$ and, for each $n \geq 0$, the complex $D^\bullet_{K_n, \Sigma, T} = D^\bullet_{K_\infty, \Sigma, T} \otimes_\bLambda^\mathbb{L} \Z_p [\cG_n]$ is represented by
$[ \Pi_n \stackrel{\phi_n}{\longrightarrow} \Pi_n ]$.\smallskip \\
This concludes the proof of Proposition \ref{propeties-iwasawa-complex}.
\end{proof}

In summary, the complexes $D^\bullet_{K_\infty, \Sigma, T}$ and $D^\bullet_{K_n, \Sigma, T}$ are represented by exact sequences
\begin{equation} \label{yoneda-Iwasawa-total}
\begin{tikzcd}
0 \arrow{r} & U_{K_\infty, \Sigma, T} \arrow{r} & \Pi_\infty \arrow{r}{\phi} & \Pi_\infty \arrow{r} & H^1 (D^\bullet_{K_\infty, \Sigma, T} ) \arrow{r} & 0
\end{tikzcd}
\end{equation}
and 
\begin{equation} \label{yoneda-finite-total}
\begin{tikzcd}
0 \arrow{r} & U_{K_n, \Sigma, T} \arrow{r} & \Pi_n \arrow{r}{\phi_n} & \Pi_n \arrow{r} & H^1 (D^\bullet_{K_n, \Sigma, T} ) \arrow{r} & 0,
\end{tikzcd}
\end{equation}
respectively. 

\subsection{Universal norms}

Let $i \geq 1$ be an integer. We recall that in \cite{BullachDaoud} the module of \textit{universal norms} of rank $i$ and level $n$ was defined as 
\[
\UN^i_n = \bigcap_{m \geq n} \NN^t_{K_m / K_n} \Big ( \bidual^i_{\Z_p [\cG_m]} U_{K_m, \Sigma, T} \Big ) 
\quad
\subseteq
\quad
\bidual^i_{\Z_p [\cG_n]} U_{K_n, \Sigma, T}.
\]

Let $v$ be a finite place of $k$ and fix a place $w$ of $K$ lying above $v$. Consider the map
\begin{equation} \label{ord-map}
\Ord_v \: U_{K_n, \Sigma, T} \to \Z_p [\cG_n], \quad a \mapsto \sum_{\sigma \in \cG_n} \ord_{w} ( \sigma a) \sigma^{-1}.
\end{equation}

\begin{lem} \label{universal-norms-p-units}
Let $i \geq 1$ and $n \geq 0$ be integers. 
\begin{liste}
\item If $\q$ denotes a prime of $k$ that is unramified but not completely split in $K_\infty / K$, then $\UN^i_n$ is contained in the kernel of
\[
\Ord_\q \: \bidual^i_{\Z_p [\cG_n]} U_{K_n, \Sigma, T} \to  \bidual^{i - 1}_{\Z_p [\cG_n]} U_{K_n, \Sigma, T}.
\]
\item The natural map $\bidual^i_\bLambda U_{K_\infty, \Sigma, T} \to \bidual^i_{\Z_p [\cG_n]} U_{K_n, \Sigma, T}$
induces an isomorphism
\[
\big ( \bidual^i_\bLambda U_{K_\infty, \Sigma, T} \big ) \otimes_\bLambda \Z_p [\cG_n] \cong \UN^i_n.
\]
\item Let $\chi \in \widehat{\cG_n}$ be a character and $\Q_p (\chi) = \Q_p ( \im \chi)$. Then
\[
\dim_{\Q_p (\chi)} e_\chi \Q_p (\chi) \UN^i_n =  \binom{|V_\chi|}{i},
\]
where $V_\chi =  \{ v \in S_\infty (k) \mid \chi ( \cG_{n, v}) = 1 \}$ and $\cG_{n, v} \subseteq \cG_n$ denotes the decomposition group at $v$. 
\end{liste}
\end{lem}

\begin{proof}
By \cite[Theorem 3.8\,(c)]{BullachDaoud} the module $\UN^i_n$ naturally identifies with a submodule of 
$\bidual^i_{\Z_p [\cG_n]} \UN^1_n$, therefore it suffices to check that $\UN^1_n$ is contained 
in the kernel of $\Ord_\q$. \\
Let $x_0 \in \UN^1_n$ and take $(x_m)_{m \geq n} \in \varprojlim_m U_{K_m, \Sigma, T}$ to be a norm-coherent sequence with $x_n$ as its bottom value. Fix a place $\mathfrak{Q}_n$ of $K_n$ lying above $\q$. Since $\mathfrak{q}$ is unramified and not completely split in $K_\infty / K_n$, we can take $N$ to be an integer big enough such that $\sigma^{-1} \mathfrak{Q}_n$ is inert in $K_\infty / K_N$, where $\sigma \in \cG_n$ is fixed. For every $m \geq n$, let $\mathfrak{Q}_m$ be a place of $K_m$ above $\sigma^{-1} \mathfrak{Q}_n$ and denote by $\ord_{\mathfrak{Q}_m} \: U_{K_m, \Sigma, T} \to \Z_p$ the map obtained from valuation at $\mathfrak{Q}_m$ by $\Z_p$-linear extension. Then for any $m \geq N$ we have 
\[
\ord_{\mathfrak{Q}_N} (x_N) = \ord_{\mathfrak{Q}_m} (\NN_{K_m / K_N} (x_m)) =
p^{m - N} \cdot \ord_{\mathfrak{Q}_m} ( x_m).
\]
This shows that the valuation of $x_N$ at $\mathfrak{Q}_N$ is infinitely divisible by $p$ and so it follows that $\ord_{\mathfrak{Q}_N} (x_N) = 0$. Hence $\ord_{\mathfrak{Q}_n} (\sigma x_n) = 0$ as well. \medskip \\
Assertion (b) is proved in the same way as \cite[Thm.\@ 3.8\,(b)]{BullachDaoud} by using Lemma \ref{silly-little-lemma}\,(b) for the analysis of the spectral sequence (15) in \textit{loc.\@ cit.} \smallskip \\
For part (c) we define a height-one prime ideal by $\p = \ker \{ \bLambda \stackrel{\chi}{\longrightarrow} \Q_p (\chi) \}$ and note that $p \not \in \p$, hence $\bLambda_\p$ is a discrete valuation ring (see \cite[\S\,3C1]{BKS2} for more details). Moreover, there is an isomorphism
\[
\faktor{\bLambda_\p}{\p \bLambda_\p} \cong \Q_p (\chi)
\]
of $\Z_p [\cG]$-modules. 
This implies, by (b), that we have an isomorphism
\begin{align} \nonumber 
e_\chi ( \Q_p (\chi) \otimes_{\Z_p} \UN^i_n) & = \Q_p (\chi) \otimes_{\Z_p [\cG_n]} \UN^i_n 
\cong \Q_p (\chi) \otimes_\bLambda \bidual^i_\bLambda U_{K_\infty, \Sigma, T}\\
&  \cong \Q_p (\chi) \otimes_{\bLambda_\p}
( \bidual^t_\bLambda U_{K_\infty, \Sigma, T}  )_\p. \label{descent-isom-characters}
\end{align}
It is therefore sufficient to compute the $\bLambda_\p$-rank of the free module $( \bidual^i_\bLambda U_{K_\infty, \Sigma, T})_\p = \exprod^i_{\bLambda_\p} ( U_{K_\infty, \Sigma, T})_\p$.\\ 
Since we assume that no finite place contained in $\Sigma$ splits completely
in $k_\infty / k$, the module $X_{K_\infty, \Sigma \setminus V_\chi}$ is $\bLambda_\p$-torsion. Note that we
also have an exact sequence
\begin{cdiagram}
0 \arrow{r} & X_{K_\infty, \Sigma \setminus V_\chi} \arrow{r} & X_{K_\infty, \Sigma} \arrow{r} & Y_{K_\infty, V_\chi} \arrow{r} & 0,
\end{cdiagram}
hence $Q ( \bLambda_\p) \otimes_\bLambda X_{K_\infty, \Sigma} = Q ( \bLambda_\p) \otimes_\bLambda Y_{K_\infty, V_\chi}$.
Further, it is well-known that $\varprojlim_n A_{\Sigma, T} (K_n)$ is $\bLambda$-torsion. Let $\cQ ( \bLambda_\p)$ be the total ring of fractions of $\bLambda_\p$, then the exact sequence (\ref{yoneda-Iwasawa}) therefore gives that
\[
\cQ ( \bLambda_\p) \otimes_\bLambda H^1 ( D^\bullet_{K_\infty, \Sigma, T} ) = \cQ ( \bLambda_\p) \otimes_\bLambda Y_{K_\infty, V_\chi}.
\]
This combines with the exact sequence (\ref{yoneda-Iwasawa-total}) to imply that
\begin{align*}
\text{rk}_{\cQ (\bLambda_\p)} ( \cQ ( \bLambda_\p) \otimes_\bLambda U_{K_\infty, \Sigma, T})  
& =
\text{rk}_{\cQ (\bLambda_\p)} (  \cQ (\bLambda_\p) \otimes_\bLambda
H^1 ( D^\bullet_{K_\infty, \Sigma, T} )) \\
& = \text{rk}_{\cQ(\bLambda_\p)} (  \cQ (\bLambda_\p) \otimes_\bLambda Y_{K_\infty, V_\chi} ) \\
& = | V_\chi|. 
\end{align*}
From this we deduce that 
\[
\mathrm{rk}_{\bLambda_\p} \exprod^i_{\bLambda_\p} ( U_{K_\infty, \Sigma, T})_\p = \text{rk}_{\cQ (\bLambda_\p)}  \exprod^i_{\cQ(\bLambda_\p)} \big ( \cQ(\bLambda_\p) \otimes_\bLambda U_{K_\infty, \Sigma, T} \big) = 
\binom{|V_\chi|}{i},
\]
as claimed. 
\end{proof}

\section{Iwasawa-theoretic congruences for Rubin--Stark elements} \label{congruences-section}

In addition to the notation introduced at the beginning of \S\,\ref{set-up-section} we also define
\[
I (\Gamma_{n})  \coloneqq \ker \big \{ \Z_p [\Gamma_{n}] \to \Z_p \big  \}
\quad \text{ and } \quad
I_{\Gamma_{n}}  \coloneqq \ker \big \{ \Z_p [\cG_{n}] \to \Z_p [\cG] \big \}.
\]
Note that $I_{\Gamma_{n}} = I (\Gamma_{n}) \cdot \Z_p [\cG_{n}]$ and hence for any $\Z_p [\cG_{n}]$-module $M$ and $i \in \N$ we have an isomorphism
\begin{equation} \label{Isomorphism}
M \otimes_{\Z_p [\cG_n]} \faktor{I_{\Gamma_n}^i}{I_{\Gamma_n}^{i + 1}} 
\cong M \otimes_{\Z_p} \faktor{I (\Gamma_n)^i}{I (\Gamma_n)^{i + 1}}.
\end{equation}
Moreover, we have an isomorphism $I (\Gamma) \coloneqq \ker \{ \Z_p \llbracket \Gamma \rrbracket \to \Z_p \} \cong \varprojlim_n I (\Gamma_{n})$.
In particular, the latter ideal is generated by $\gamma - 1$ for any topological generator $\gamma \in \Gamma$, and there is an isomorphism
\[
\faktor{I ( \Gamma)^i}{I (\Gamma)^{i + 1}} \stackrel{\simeq}{\longrightarrow} \Gamma, \quad (\gamma - 1)^i \mapsto \gamma. 
\]

\subsection{Darmon derivatives}

Suppose the $p$-part of the Rubin--Stark Conjecture \ref{Rubin--Stark-conj} holds true for all extensions $K_n / k$ and the data $(V, \Sigma, T)$. Then \cite[Prop.\@ 6.1]{Rub96} (see also \cite[Prop.~3.5]{Sano}) implies that the family $\varepsilon_{K_\infty / k, \Sigma, T} \coloneqq (\varepsilon^V_{K_n / k, \Sigma, T})_{n \geq 0}$ defines an element of $\varprojlim_{n \in \N_0} \bidual^r_{\Z_p [\cG_n]} U_{K_n, \Sigma, T} = \bidual^r_\bLambda U_{K_\infty, \Sigma, T}$. \smallskip \\
Let $W \coloneqq V' \setminus V$ and $e \coloneqq | W|$. 

\begin{conjecture} \label{OrderOfVanishingConjecture}
One has 
$\varepsilon_{K_\infty / k, \Sigma, T} \in I_{\Gamma}^e \cdot \bidual^r_\bLambda U_{K_\infty, \Sigma, T}$.
\end{conjecture}

\begin{rk}
\begin{liste}
\item Variants of Conjecture \ref{OrderOfVanishingConjecture} have previously appeared in the literature in many places, with its archetypical relative being the `guess' formulated by Gross for the Euler system of Stickelberger elements \cite[top of p.\,195]{gross88}. A version for arbitrary rank was then  formulated in \cite{Burns07}, see also \cite[Conj.\@ 4]{Sano}.  
\item We will see below that Conjecture \ref{OrderOfVanishingConjecture} is implied by a (relevant variant of a) Iwasawa Main Conjecture and this is indeed already well-known (see Remark \ref{order-of-vanishing-evidence-rk} for more details). A direct proof by analytic means for the Euler system of Stickelberger elements is given in \cite{DasguptaSpiess}.  
\item A containment as in the statement of Conjecture \ref{OrderOfVanishingConjecture} should be thought of as an order of vanishing statement. In fact, it can be directly linked to the order of vanishing of a $p$-adic $L$-function in many cases. 
\end{liste}
\end{rk}

Recall that any element $a \in \bidual^r_\bLambda U_{K_\infty, \Sigma, T}$ is by definition a morphism $a \: \exprod^r_\bLambda (U_{L_\infty, \Sigma, T})^\ast \to \bLambda$. In particular, $\im (a)$ is a well-defined ideal of $\bLambda$.

\begin{prop} \label{order-of-vanishing-prop}
Assume that, for every $n \geq 0$, the $p$-part of the Rubin--Stark Conjecture holds for the extension $K_n / k$ and the data $(V, S, T)$. 
\begin{liste}
\item Conjecture \ref{OrderOfVanishingConjecture} holds true if one has $\im (\varepsilon^V_{K_\infty / k, \Sigma, T}) \subseteq \Fitt^0_\bLambda (Y_{K_\infty, W})$.
\item If $W$ contains at most one place which ramifies in $k_\infty / k$, then Conjecture \ref{OrderOfVanishingConjecture} holds true. 
\end{liste}
\end{prop}

\begin{proof}
Observe that we have a surjection
$Y_{K_\infty, W} \twoheadrightarrow Y_{K, W} = \bigoplus_{v \in W} \Z_p$,
hence $\Fitt^0_\bLambda ( Y_{K_\infty, W}) \subseteq \Fitt^0_\bLambda \big (\bigoplus_{v \in W} \Z_p \big ) = I_\Gamma^e$.
We have an inclusion (see the proof of \cite[Lem.\@ 2.7\,(c)]{BullachDaoud})
\[
\big \{ f (\eta_{K_\infty}) \mid f \in \exprod^r_\bLambda \Pi_\infty^\ast \big \} \subseteq \im (\eta_{K_\infty}), 
\]
therefore $(\exprod_{i \in I} b_i^\ast) (\varepsilon^V_{K_\infty / k, \Sigma, T}) \in I_\Gamma^e$ for any index set $I \subseteq \{1, \dots, d \}$. Since the elements of the form $\exprod_{i \in I} b_i$ form a $\bLambda$-basis of $\exprod^r_\bLambda \Pi_\infty$, we deduce that $\varepsilon^V_{K_\infty / k, \Sigma, T} \in I_\Gamma^e \cdot \exprod^r_\bLambda \Pi_\infty$. The validity of Conjecture \ref{OrderOfVanishingConjecture} now follows from Lemma \ref{useful-lem} below.\smallskip\\
To prove (b) it now suffices to verify the inclusion assumed in (a) under the stated condition. The proof of \cite[Lem.\@ 6.3\,(b)]{scarcity} shows that $\im (\varepsilon^V_{K_\infty / k, \Sigma, T})$ is contained in $\Fitt^0_\bLambda (X_{K_\infty, \Sigma \setminus V})$ and so this follows from the inclusion $\Fitt^0_\bLambda (X_{K_\infty, \Sigma \setminus V}) \subseteq\Fitt^0_\bLambda (Y_{K_\infty, W})$ induced by the surjection $X_{K_\infty, \Sigma \setminus V} \twoheadrightarrow Y_{K_\infty, W}$.
\end{proof}

\begin{lem} \label{useful-lem}
Let $u \in \bidual^r_\bLambda U_{K_\infty, \Sigma, T}$ be a norm-coherent sequence. Then
\[
u \in I_{\Gamma}^e \cdot \bidual^r_\bLambda U_{K_\infty, \Sigma, T}
\iff
u \in I_{\Gamma}^e \cdot \exprod^r_{\bLambda} \Pi_\infty.
\]
\end{lem}

\begin{proof}
By \cite[Lem.\@ B.12]{Sakamoto20}, the exact sequence (\ref{yoneda-Iwasawa-total})
induces an exact sequence 
\begin{equation} \label{ryotaro-sequence}
\begin{tikzcd}
0 \arrow{r} & \bidual^r_\bLambda U_{K_\infty, \Sigma, T} \arrow{r} &\exprod^r_\bLambda \Pi_\infty \arrow{r}{\phi} &\Pi_\infty \otimes_\bLambda \exprod^{r-1}_\bLambda \Pi_\infty,
\end{tikzcd}
\end{equation}
whence the implication `$\Rightarrow$' is clear. 
If we now fix a topological generator $\gamma \in \Gamma$, then $u = (\gamma - 1)^e \kappa$ for some $\kappa \in \exprod^r_{\bLambda} \Pi_\infty$ implies that
\[
0 = \phi (u) = \phi ( (\gamma - 1)^e \kappa) = (\gamma - 1)^e  \cdot \phi (\kappa), 
\]
hence $\phi (\kappa) = 0 $ since $(\gamma - 1)^e \in \bLambda$ is not a zero divisor by Lemma \ref{silly-little-lemma} and $\Pi_\infty \otimes_\bLambda \exprod^{r-1}_\bLambda \Pi_\infty$ is torsion-free. The exact sequence (\ref{ryotaro-sequence}) thus reveals that $\kappa \in  \bidual^r_\bLambda U_{K_\infty, \Sigma, T}$. 
\end{proof}

\begin{rk} \label{order-of-vanishing-evidence-rk}
 It is expected that the family of Rubin--Stark elements 
 $\varepsilon_{K_\infty / k, \Sigma, T}$
 gives rise to the following higher-rank `Iwasawa Main Conjecture'-type `divisibility' of reflexive hulls
    \begin{equation} \label{divisibility-2}
    \im (\varepsilon_{K_\infty / k, \Sigma, T})^{\ast \ast} \subseteq \Fitt^0_\bLambda (H^2_{T, \Iw} ( \bigO_{
    K, \Sigma}, \Z_p (1)))^{\ast \ast} , 
    \end{equation}
    where the $T$-modified Iwasawa-cohomology group $H^2_{T, \Iw} ( \bigO_{K, \Sigma}, \Z_p (1))$ (see, for example, \cite[\S\,3.1]{BullachDaoud} for a definition) sits in a short exact sequence
    \begin{equation} \label{Iwasawa-cohomology}
    \begin{tikzcd}
    0 \arrow{r} & \varprojlim_n A_{\Sigma, T} ( K_n) \arrow{r} & H^2_{T, \Iw} ( \bigO_{K, \Sigma}, \Z_p (1)) \arrow{r} & X_{K_\infty, \Sigma \setminus V} \arrow{r} & 0. 
    \end{tikzcd}
    \end{equation}
    Our assumption $V' \subsetneq \Sigma$ therefore implies that we have a surjection $H^2_{T, \Iw} ( \bigO_{K, \Sigma}, \Z_p (1)) \twoheadrightarrow Y_{K_\infty, W}$
    and so it follows from (\ref{divisibility-2}) that $\im (\eta_{K_\infty})^{\ast \ast} \subseteq \Fitt_\bLambda^0 ( Y_{K_\infty, W})^{\ast \ast} = \Fitt_\bLambda^0 ( Y_{K_\infty, W})$.
    This shows that the inclusion assumed in Proposition \ref{order-of-vanishing-prop}\,(a) would be a consequence of one divisibility in a relevant Iwasawa Main Conjecture.
\end{rk}

\begin{definition}
Assume Conjecture~\ref{OrderOfVanishingConjecture} holds. The (Iwasawa-theoretic) \emph{Darmon derivative} of $\varepsilon_{K_\infty /k, \Sigma, T}$ with respect to a topological generator $\gamma \in \Gamma$ is the bottom value $\kappa_0$ of the unique norm-coherent sequence $\kappa = (\kappa_n)_n$ with the property $(\gamma - 1)^e \cdot \kappa = \varepsilon_{K_\infty /k, \Sigma, T}$. 
\end{definition}

\begin{rk}
\begin{liste}
\item If $k = \Q$ (in which case the appearing Rubin--Stark elements are cyclotomic units, see Example \ref{Rubin--Stark-examples}\,(a)), the notion of Darmon derivative recovers the element considered by Solomon in \cite{Solomon1992} (cf.\@ the proof of Proposition \ref{MRS-equivalent} below). Lemma~\ref{universal-norms-p-units}\,(a) moreover provides an easy proof for the analogue of \cite[Prop.~2.2\,(i)]{Solomon1992}, namely that the valuation of $\kappa_0$ at a prime coprime to $p$ is almost always trivial. In contrast, we will see later that, as first observed by Solomon, the valuation of $\kappa_0$ at a prime above $p$ can encode important arithmetic information. 
\item The question if the Darmon derivative $\kappa_0$ vanishes is related to information about class groups. 
To explain this in a little more detail, we assume `the higher-rank Iwasawa Main Conjecture' formulated by Burns, Kurihara, and Sano in \cite[Conj.\@ 3.1]{BKS2} holds true. One can then show that
\begin{align*}
 e_{r'} \Q_p [\cG] \otimes_{\Z_p [\cG]} (A_{\Sigma, T} (K_\infty))^\Gamma = 0
& \iff \kappa_0 \neq 0, 
\end{align*}
where $e_{r'} \coloneqq \sum_\chi e_\chi$ is the sum over all primitive orthogonal idempotents $e_\chi$ associated with characters $\chi \in \widehat{\cG}$ for which one has $e_\chi \varepsilon_{K / k, \Sigma, T}^{V'} \neq 0$.
That is, under condition (F) (see \ref{condition-F}) the containment in Conjecture \ref{OrderOfVanishingConjecture} is `optimal'.
\item Our terminology follows \cite{bks-kato-euler} where an element defined via a closely related construction is referred to as the \textit{Iwasawa-Darmon derivative} (see \cite[Def.~4.6]{bks-kato-euler}). This points to Darmon \cite{Darmon} who first interpreted this construction as a derivative process (see also the discussion in \cite[Rk.~4.8]{Sano}). 
\end{liste}
\end{rk}

\subsection{The conjecture of Mazur--Rubin and Sano}

Recall that by \cite[Lem.\@ 2.7]{scarcity} (see also \cite[Lem.~2.11]{Sano}) for every $n \in \N$ there is an isomorphism
\[
\bidual^r_{\Z_p [\cG]} U_{K, \Sigma, T} \stackrel{\simeq}{\longrightarrow} \Big ( \bidual^r_{\Z_p [\cG_n]} U_{K_n, \Sigma, T} \Big )^{\Gamma_n} \subseteq  \bidual^r_{\Z_p [\cG_n]} U_{K_n, \Sigma, T} 
\]
which gives rise to an injection
\[
\nu_n \: \Big ( \bidual^r_{\Z_p [\cG]} U_{K, \Sigma, T} \Big ) \otimes_{\Z_p} \faktor{I (\Gamma_n)^e}{I (\Gamma_n)^{e + 1}} \hookrightarrow \Big ( \bidual^r_{\Z_p [\cG_n]} U_{K_n, \Sigma, T} \Big ) \otimes_{\Z_p} \faktor{\Z_p [\Gamma_n]}{I (\Gamma_n)^{e + 1}}.
\]
We note that this injection satisfies
\begin{equation} \label{injection-norm-property}
\nu_n ( \NN_{\gal{K_n}{K}}^r a \otimes x) = (\NN_{\gal{K_n}{K}} a) \otimes x
\end{equation}
for any $a \in  \bidual^r_{\Z_p [\cG_n]} U_{K_n, \Sigma, T}$ and $x \in \faktor{I (\Gamma_n)^e}{I (\Gamma_n)^{e + 1}}$, see \cite[Rk.~2.12]{Sano}. Finally, we define \textit{Darmon's twisted norm operator} 
\[
\mathcal{N}_n \:  \bidual^r_{\Z_p [\cG]} U_{K_{n}, \Sigma, T} \to \Big ( \bidual^r_{\Z_p [\cG]} U_{K, \Sigma, T} \Big ) \otimes_{\Z_p} \faktor{\Z_p [\Gamma_n]}{I (\Gamma_n)^{e + 1}}, 
\quad
a \mapsto \sum_{\sigma \in \Gamma_n} \sigma a \otimes \sigma^{-1}.
\]

\begin{lem} \label{NormCalculation}
Fix a topological generator $\gamma \in \Gamma$ and let $u, \kappa \in \bidual^r_\bLambda U_{K_\infty, \Sigma, T}$ be norm-coherent sequences satisfying $u = (\gamma - 1)^e \kappa$. Then we have
\[
\mathcal{N}_n ( u_n) = \nu_n ( \kappa_0 \otimes (\gamma - 1)^e ) 
\quad \text{ for all } n  \in \N. 
\]
\end{lem}

\begin{proof}
We calculate:
\begin{align*}
    \mathcal{N}_n (u_n ) & = \mathcal{N}_n ( (\gamma - 1)^e \kappa_n)  \\
    & = \sum_{\sigma \in \Gamma_n} \sigma (\gamma  - 1)^e \kappa_n \otimes \sigma^{-1}  \\
    & = \sum_{\sigma \in \Gamma_n} \sigma \kappa_n \otimes \sigma^{-1} (\gamma - 1)^e  \\
    & = \sum_{\sigma \in \Gamma_n} \sigma \kappa_n \otimes (\gamma - 1)^e  \\
    & = ( \NN_{\gal{K_n}{K}} \kappa_n) \otimes (\gamma - 1)^e.
\end{align*}
Here the third equality from the bottom is obtained by reparametrising the sum and the penultimate equality follows from
\[
\sigma^{-1} (\gamma - 1)^e - (\gamma - 1)^{e} = (\sigma^{-1} - 1)(\gamma - 1)^e \equiv 0 \mod I(\Gamma_n)^{e + 1}.
\]
The property (\ref{injection-norm-property}) then yields
\[
\nu_{n} (\kappa_0 \otimes (\gamma - 1)^e) = \nu_{n} ( \NN_{K_n / K}^r ( \kappa_n) \otimes (\gamma - 1)^e) = (\NN_{\gal{K_n}{K}}\kappa_n) \otimes (\gamma - 1)^e.
\]
and this finishes the proof of the Lemma. 
\end{proof}

\begin{rk}
Lemma \ref{NormCalculation} implies that, in particular, the element $\kappa_0 \otimes (\gamma - 1)^e \in \bidual^r_{\Z_p [\cG]} U_{K, \Sigma, T}
\otimes_{\Z_p} \faktor{I (\Gamma)^e}{I (\Gamma)^{e + 1}}$ does not depend on the choice of topological generator $\gamma$ if $u$ is fixed.
\end{rk}

Let $v \in W$ and recall that we have fixed a place $w$ of $K$ above $v$. Denote by $\Gamma_w$ the decomposition group of $w$ inside $\Gamma$ and write 
\[
\mathrm{rec}_w \: K^\times \hookrightarrow K^\times_w \to \Gamma_w \subseteq \Gamma
\]
for the local reciprocity map at $w$, where $K_w$ denotes the completion of $K$ at $w$. 
Consider the map
\[
\Rec_v \: K^\times \to \faktor{I_{\Gamma}}{I_{\Gamma}^2}, 
\quad
a \mapsto \sum_{\sigma \in \cG} (\mathrm{rec}_w (\sigma a) - 1) \sigma^{-1},
\]
which by \cite[Prop.~2.7]{Sano} and (\ref{Isomorphism}) induces the map
\begin{equation} \label{Rec-map}
\Rec_W = \bigwedge_{v \in W} \Rec_v \: \bidual^{r'}_{\Z_p [\cG]} U_{K, \Sigma, T} \to \left ( \bidual^r_{\Z_p [\cG]} U_{K, \Sigma, T} \right ) \otimes_{\Z_p} \faktor{I (\Gamma)^e}{I (\Gamma)^{e + 1}}.
\end{equation}
In addition, we define an idempotent $e_r \in \Q_p [\cG]$ as the sum of all primitive orthogonal idempotents $e_{\chi}$ for characters $\chi \in \widehat{\cG}$ such that $e_{\chi} \Q_p (\chi) U_{K, \Sigma \setminus W, T}$ has $\Q_p (\chi)$-dimension $r$. \\
With this notation in place, there is an isomorphism
\[
\Ord_W = \bigwedge_{v \in W} \Ord_v \: e_r\Q_p \bidual^{r'}_{\Z_p [\cG]} U_{K, \Sigma, T} \stackrel{\simeq}{\longrightarrow} e_r \Q_p \bidual^r_{\Z_p [\cG]} U_{K, \Sigma \setminus W, T},
\]
where $\Ord_v$ is the map defined in (\ref{ord-map}), see \cite[Lem.~5.1]{Rub96}.

\begin{conjecture}[Iwasawa-theoretic Mazur--Rubin--Sano] \label{MRS}
Assume that, for all $n \geq 0$, the $p$-part of the Rubin--Stark Conjecture \ref{Rubin--Stark-conj} holds for the extension $K_n / k$ and the data $(V, \Sigma, T)$. Then, there exists an element 
\begin{align*}
\mathfrak{k} = (\mathfrak{k}_n)_{n \in \N_0} & \in \left ( \bidual^r_{\Z_p [\cG]} U_{K, \Sigma, T} \right ) \otimes_{\Z_p}  \faktor{I (\Gamma)^e}{I (\Gamma)^{e + 1}} \\
& = \left ( \bidual^r_{\Z_p [\cG]} U_{K, \Sigma, T} \right ) \otimes_{\Z_p}  \varprojlim_n \faktor{I (\Gamma_n)^e}{I (\Gamma_n)^{e + 1}}
\end{align*}
such that $\nu_n (\mathfrak{k}_n) = \mathcal{N}_n (\varepsilon^V_{K_n / k, \Sigma, T})$ for all $n$, and 
\[
\mathfrak{k} =  (-1)^{re} \cdot \Rec_W ( \varepsilon^{V'}_{K / k, \Sigma, T}),
 \]
 where the equality takes place in
$
\Q_p \left ( \bidual^r_{\Z_p [\cG]} U_{K, \Sigma, T} \right ) \otimes_{\Z_p}  \faktor{I (\Gamma)^e}{I (\Gamma)^{e + 1}}
$.
\end{conjecture}

\begin{rk}
\begin{liste}
\item The above conjecture is taken from \cite[Conj.~4.2]{BKS2} and is an Iwasawa-theoretic version of a conjecture that was independently proposed by Mazur--Rubin \cite{MazurRubin} and Sano \cite{Sano}. The latter of which, in turn, unify the central conjectures in \cite{Burns07} and \cite{Darmon}. 
\item Conjecture \ref{MRS} is known in the following cases:
\begin{itemize}
\item $k = \Q$ and $K$ is totally real, in this case the conjecture follows from a classical result of Solomon \cite{Solomon1992} (see \cite[Thm.~4.10]{BKS2}),
\item $k$ is totally real and $K$ is CM, in this case the conjecture follows from the validity of the Gross--Stark conjecture that has been settled by Dasgupta, Kakde and Ventullo in \cite{DKV} (see \cite[Thm.~4.9]{BKS2}).
\end{itemize}
\end{liste}
\end{rk}

In the remainder of this section we will give a more explicit version of Conjecture~\ref{MRS} in cases where Conjecture \ref{OrderOfVanishingConjecture} holds true. 

\begin{prop} \label{MRS-equivalent}
The following assertions are equivalent:
\begin{enumerate}[label=(\roman*)]
\item Conjecture~\ref{MRS} is valid for the data $(k_\infty / k, K, S, T, V')$,
\item Conjecture~\ref{OrderOfVanishingConjecture} holds for $(k_\infty / k, K, S, T, V')$ and we have an equality
\begin{equation} \label{Statement1}
 \kappa_0 \otimes (\gamma - 1)^e = (-1)^{re} \cdot \Rec_W ( \varepsilon^{V'}_{K / k, \Sigma, T}),
\end{equation}
where $\kappa_0$ denotes the Darmon derivative of $\varepsilon_{K_\infty / k, \Sigma, T}$ with respect to a fixed topological generator $\gamma \in \Gamma$.
\end{enumerate}
\end{prop}

\begin{proof}
Let us first assume that statement (ii) holds. In light of Lemma \ref{NormCalculation},  the element 
$\mathfrak{k} = (\mathfrak{k}_n)_n$ given by $\mathfrak{k}_n =  \nu_n( \kappa_0 \otimes (\gamma - 1)^e) \in \big (  \bidual^r_{\Z_p [\cG_n]} U_{K_n, \Sigma, T} \big) \otimes_{\Z_p} \faktor{I (\Gamma_n)^e}{I (\Gamma_n)^{e + 1}} $
satisfies the requirements of Conjecture \ref{MRS}. 
\smallskip \\
Conversely, suppose that Conjecture \ref{MRS} holds true for $\eta$. 
By assumption $\nu_n (\mathfrak{k}_n) = \mathcal{N}_n (\eta_{K_n})$ for all $n$, so it follows from \cite[Prop.~4.17]{BKS} that $\varepsilon^V_{K_n / k, \Sigma, T} \in I_{\Gamma_n}^e \cdot \exprod^r_{\Z_p [\cG_n]} \Pi_n$. Thus, we can write $\varepsilon^V_{K_n / k, \Sigma, T} = (\gamma - 1)^e x_n$ for some $x_n \in \exprod^r_{\Z_p [\cG_n]} \Pi_n$. 
The element $x_n$ defines a unique class modulo $\big ( \exprod^r_{\Z_p [\cG_n]} \Pi_n \big)^{\Gamma_n}$, hence $(x_n)_n$ is a norm-compatible sequence modulo these modules. Observe that in the commutative diagram of transition maps
\begin{cdiagram}[column sep=small]
0 \arrow{r} & \exprod^r_{\Z_p [\cG]} \Pi_0 \arrow{r} \arrow{d} &  \exprod^r_{\Z_p [\cG_{n + 1}]} \Pi_{n + 1}
\arrow{r} \arrow{d}{ \NN_{K_{n + 1} / K_n}}& 
\displaystyle
\faktor{\Big ( \exprod^r_{\Z_p [\cG_{n + 1}]} \Pi_{n + 1} \Big)}{
\Big ( \exprod^r_{\Z_p [\cG_{n + 1}]} \Pi_{n + 1} \Big)^{\Gamma_{n + 1}}
} \arrow{d}
\arrow{r} & 0 \\
0 \arrow{r} & \exprod^r_{\Z_p [\cG]} \Pi_0 \arrow{r} &  \exprod^r_{\Z_p [\cG_n]} \Pi_{n}
\arrow{r} & 
\displaystyle
\faktor{\Big( \exprod^r_{\Z_p [\cG_n]} \Pi_{n} \Big)}{
\Big ( \exprod^r_{\Z_p [\cG_n]} \Pi_{n} \Big)^{\Gamma_{n}}
}
\arrow{r} & 0
\end{cdiagram}
the vertical map on the left is multiplication by $p$. Taking inverse limits (these are all finitely generated $\Z_p$-modules, so compact and therefore taking limits is exact), we get an isomorphism
\[
\varprojlim_{n \geq 0}  \exprod^r_{\Z_p [\cG_n]} \Pi_{n} \cong \varprojlim_{n \geq 0}\faktor{\big( \exprod^r_{\Z_p [\cG_n]} \Pi_{n} \big)}{
\big ( \exprod^r_{\Z_p [\cG_n]} \Pi_{n} \big)^{\Gamma_{n}}
}.
\]
Consequently, the family $(x_n)_{n \geq 0}$ can be regarded as an element of $\varprojlim_{n \geq 0}  \exprod^r_{\Z_p [\cG_n]} \Pi_{n} = \exprod^r_\bLambda \Pi$. By construction, we have 
\[
(\gamma - 1)^e \cdot (x_n)_n = (\eta_{K_\infty, n})_n \in
\bidual^r_{\bLambda} U_{K_\infty, \Sigma, T}. 
\]
and Lemma \ref{useful-lem} gives $(x_n)_n \in \bidual^r_\bLambda U_{K_\infty, \Sigma, T}$. 
\smallskip \\
For the second part of (ii) it suffices to note that Lemma \ref{NormCalculation} now implies that 
for $n$ big enough we have
\[
\nu_n ( x_0 \otimes (\gamma - 1)^e) =  
\mathcal{N}_n ( \varepsilon^V_{K_n / k, \Sigma, T}) = \nu_n (\mathfrak{k}_n),
\]
so $x_0 \otimes (\gamma - 1)^e = \mathfrak{k}_n$ by the injectivity of $\nu_n$. 
\end{proof}

The appearance of $\Ord_W^{-1}$ in (\ref{Statement1}) suggests the application of the map $\Ord_W$ to said equation. This is possible in cases where $r \geq e$ and we summarise this observation in the following Lemma.

\begin{lem} \label{MRS-Solomon-formulation}
Assume that $r \geq e$. The equality (\ref{Statement1}) holds if and only if the equality
\begin{equation} \label{Statement2}
   \Ord_W ( \kappa_0) \otimes (\gamma - 1) = (-1)^e \cdot \Rec_W (\varepsilon^V_{K, \Sigma \setminus W, T}),
\end{equation}
holds in $\Q_p \left ( \bidual^{r - 1}_{\Z_p [\cG]} U_{K, \Sigma \setminus W, T} \right ) \otimes_{\Z_p}
\faktor{I (\Gamma)}{I (\Gamma)^2}$. 
\end{lem}

\begin{proof}
Let us first show that each side of the conjectural equality (\ref{Statement1}) can be considered as an element of the module
\[
 e_{S_\infty, r} \C_p \UN^r_0 \otimes_{\Z_p} \faktor{I (\Gamma)}{I(\Gamma)^2},
\]
where $e_{S_\infty, r} \coloneqq \sum_\chi e_\chi$ with $\chi$ ranging over all characters such that $|V_\chi| = r$ for the set $V_\chi$ defined in Lemma \ref{universal-norms-p-units}\,(c). 
As of the left hand side of (\ref{Statement1}), it suffices to prove that $\kappa_0$ is contained in 
$e_{S_\infty, r} \Q_p \UN^r_0$ and this will follow (by means of the isomorphism (\ref{descent-isom-characters})) if we can show that $\kappa = (\kappa_n)_n$ is not supported at any height-one prime ideal $\p$ of $\bLambda$ of the form $\ker \{ \bLambda \stackrel{\chi}{\to} \Q_p (\chi)\}$ for a character $\chi$ such that $|V_\chi | > r$. 
To this end, we recall that, by Lemma \ref{silly-little-lemma}, $\gamma - 1$ is a non-zero divisor in $\bLambda$ and, also, that one has
\[
\bLambda_\p \cdot \varepsilon_{K_\infty / k, \Sigma, T}  = 
\bLambda_\p \cdot ( \NN_{K_n / K_{\chi, n}} (\varepsilon^V_{K_n / k, \Sigma, T}))_n  
= \bLambda_\p \cdot ( \varepsilon^V_{K_{\chi, n} / k, \Sigma, T})_n = 0,
\]
where $K_{\chi, n} \coloneqq K_\chi k_n$ with $K_\chi$ the kernel field of $\chi$ and the last equality holds because by assumption more than $r$ places split completely in $K_{\chi, n} / k$. This shows the claim for the left hand side of (\ref{Statement1}). \\
If $e_r \Rec_W$ is the zero map, then, since $\varepsilon^V_{K, \Sigma \setminus W, T}$ belongs to the $e_r$-component, this is also clear for the right hand side of (\ref{Statement1}). If $e_\chi \Rec_W$ is not the zero map for some primitive orthogonal idempotent $e_\chi$ appearing as a summand of $e_r$, then \cite[Lem.\@ 4.2]{BKS} shows that the map
\[
\Rec_{\p, \chi} \: e_\chi \Q_p (\chi) U_{K, \Sigma, T} \to \Q_p (\chi) \otimes_{\Z_p}  \faktor{I (\Gamma)}{I(\Gamma)^2}, \quad 
a \mapsto \sum_{\sigma \in \cG} \chi (\sigma) \otimes (\rec_\p (\sigma^{-1} a) - 1) 
\]
is surjective. Now, $e_\chi \Q_p (\chi) \UN^1_0$ is contained in $\ker \Rec_{\p, \chi}$ and comparing dimensions (using Lemma \ref{universal-norms-p-units}\,(c)) gives $e_\chi \Q_p (\chi) \UN^1_0 = \ker \Rec_{\p, \chi}$. By \cite[Lem.\@ 4.2]{BKS} we then have
\begin{align*}
&  \im \Big \{ e_\chi \Q_p (\chi) \exprod^{r'}_{\Z_p [\cG]} U_{K, \Sigma, T} \stackrel{\Rec_W}{\longrightarrow} 
e_\chi \Q_p (\chi) \exprod^{r}_{\Z_p [\cG]} U_{K, \Sigma, T} \otimes_{\Z_p} \faktor{I (\Gamma)}{I(\Gamma)^2} \Big \} \\
& = e_\chi \Q_p (\chi) \exprod^{r}_{\Z_p [\cG]} \UN^1_0 \otimes_{\Z_p} \faktor{I (\Gamma)}{I(\Gamma)^2} 
 = e_\chi \Q_p (\chi) \UN^r_0 \otimes_{\Z_p} \faktor{I (\Gamma)}{I(\Gamma)^2}, 
\end{align*}
where the second equality is by \cite[Thm.\@ 3.8\,(c)]{BullachDaoud}. This finishes the proof of the claim for right hand side of (\ref{Statement1}). \smallskip \\
By scalar extension, the map $\Ord_W$ induces a map
\begin{equation} \label{Ord-2-map}
\Big ( \C_p \exprod^{r}_{\Z_p [\cG]} \UN^1_0 \Big ) \otimes_{\Z_p}
\faktor{I (\Gamma)}{I (\Gamma)^{2}} \to \Big ( \C_p \exprod^{r - 1}_{\Z_p [\cG]} \UN^1_0 \Big ) \otimes_{\Z_p}
\faktor{I (\Gamma)}{I (\Gamma)^{2}}
\end{equation}
which we also denote by $\Ord_W$. Since $\Ord_W$ is injective on $e_{S_\infty, r} \C_p \UN^r_0$ by Lemma \ref{Ord-injective-UN} below and $\faktor{ I ( \Gamma)}{I (\Gamma)^2} \cong \Gamma$ is isomorphic to  $\Z_p$, the map (\ref{Ord-2-map}) is injective on the $e_{S_\infty, r}$-component as well. Thus, the equation (\ref{Statement1}) holds if and only if 
\[
\Ord_W ( \kappa_0) \otimes (\gamma - 1) = ( \Ord_W \circ \Rec_W ) (\varepsilon^{V'}_{K / k, \Sigma, T})
\]
holds. Now, by virtue of (\ref{biduals-duals-hom}) being a homomorphism, we have
\[
 \Ord_W \circ \Rec_W = (-1)^{e^2} \cdot (\Rec_W \circ \Ord_W) = (-1)^e \cdot (\Rec_W \circ \Ord_W)
\]
and so the Lemma follows 
by combining this with the fact that
\begin{equation} \label{RS-elements-ord-map}
\Ord_W ( \varepsilon^{V'}_{K / k, \Sigma, T}) = (-1)^{er} \cdot \varepsilon^V_{E / k, \Sigma \setminus W, T},
\end{equation}
which holds by \cite[Prop.~3.6]{Sano}. 
\end{proof}

\begin{lem} \label{Ord-injective-UN}
Assume $r \geq e$. The map $\Ord_W$ restricts to an injection
\[
\Ord_W \: e_{S_\infty, r} \C_p \UN^r_0 \hookrightarrow \C_p \exprod^{r - 1}_{\Z_p [\cG]} U_{K, \Sigma \setminus W, T}.
\]
\end{lem}

\begin{proof}
Let $\chi \in \widehat{\cG}$ be a character such that $e_\chi$ is a summand of $e_{S_\infty, r}$ which, by Lemma \ref{universal-norms-p-units}\,(c), implies that $e_\chi \Q_p (\chi) \UN^1_0$ has $\Q_p (\chi)$-dimension $r$. Thus, 
\[
e_\chi (\bigoplus_{\p \in W} \Ord_\p) \: e_\chi \Q_p (\chi) \UN^1_0 \to e_\chi \Q_p (\chi) Y_{K, W} \cong \bigoplus_{\p \in W} \Q_p (\chi)
\]
is a map of $r$-dimensional $\Q_p (\chi)$-vector spaces. 
Now, by \cite[Thm.\@ 3.8\,(c)]{BullachDaoud}, we have $\C_p \UN^r_0 = \C_p \exprod^r_{\Z_p [\cG]} \UN^1_0$ and this combines with the previous discussion and \cite[Lem.\@ 4.2]{BKS} to imply that
the Lemma will follows if we can prove that the map $e_\chi (\bigoplus_{\p \in W} \Ord_\p)$ is surjective.\\ 
Let $\p$ be a place in $W$ and choose an integer $n$ such that $\p$ is totally ramified in $K_\infty / K_n$. Fix a prime $\mathfrak{P}_n$ of $K_n$ lying above $\p$ and, for $m \geq n$, write $\mathfrak{P}_m$ for the unique prime of $K_m$ lying above $\mathfrak{P}_n$. If $h$ denotes the class number of $K_n$, then $\mathfrak{P}_n^h$ is a principal ideal generated by $x$, say. We then have $\NN_{K_m / K_n} (\mathfrak{P}_m) = \mathfrak{P}_n$ for all $m \geq n$, hence $\NN_{K_n / K}(x) \in \NN_{K_m / K} ( \bigO_{K_m, \Sigma}^\times)$ for all integers $m \geq 0$. By a standard compactness argument (see, for example, \cite[Lem.\@ 3.10]{BullachDaoud}) we therefore have $\NN_{K_n / K}(x) \in \Q_p \UN^1_0$. \\
By construction, $\NN_{K_n / K}(x)$ is a generator of the ideal $(\mathfrak{P}_n \cap \bigO_K)^{h f}$, where $f$ is the residual degree of $\p$ in $K_n / K$. From this we see that the image of $\NN_{K_n / K}(x)$ under the map $e_\chi (\bigoplus_{\p \in W} \Ord_\p)$ generates the copy of $\Q_p (\chi)$ for $\p$. Since $\p$ was chosen to be an arbitrary place in $W$, this proves the Lemma. 
\end{proof}

\section{The Gross--Kuz'min conjecture and condition (F)} \label{big-gross-kuzmin-section}

In this section we will investigate a conjecture 
due to Gross \cite{Gross} and, independently, Kuz'min \cite{Kuzmin}. 

\subsection{Coinvariants of class groups}
\label{gross-kuzmin-section}

We resume the notation of \S\,\ref{set-up-section}. In particular, $k_\infty / k$ denotes a $\Z_p$-extension in which no finite place contained in $\Sigma$ splits completely, $K_\infty = K k_\infty$ and
\[
A_{M, T} ( K_\infty) = \varprojlim_{n \geq 0} A_{M, T} (K_n)
\quad \text{ for any } M \supseteq S_\infty (k).
\]
If $M = S_\infty(k)$ or $T = \emptyset$, we will suppress the respective set in the notation. 

\begin{conjecture}[Gross--Kuz'min] \label{gross-kuzmin-conjecture}
If $K_\infty^\cyc / K$ is the cyclotomic $\Z_p$-extension, then the module of $\Gamma$-coinvariants 
$
( A_{\Sigma} (K_\infty^\cyc) )_{\Gamma}
$
is finite. 
\end{conjecture}

\begin{rk}
\begin{liste}
\item It is necessary to work with $\Sigma$-class groups in this context because in general it is \textit{not} true that the $\Gamma$-coinvariants of $A (K^\cyc_\infty)$ are finite (see \cite[Prop.~1.17]{Kol91} and \cite[Prop.~2]{Greenberg73} for examples). However, it is known to be true if there is only one prime above $p$ in $K$ or $K$ is totally real and Leopoldt's conjecture holds for $K$. 
\item We remind the reader that for any finitely generated $\Z_p \llbracket \Gamma \rrbracket$-module $M$ the module $M_\Gamma$ is finite if and only if $M^\Gamma$ is finite (see, for example, \cite[App.~A.2, Prop.\@ 2]{coates-sujatha-06}). Conjecture~\ref{gross-kuzmin-conjecture} can therefore also formulated as the statement that $( A_{\Sigma} (K_\infty^\cyc) )^{\Gamma}$ is finite. 
\end{liste}
\end{rk}

We follow Burns, Kurihara and Sano in considering the following condition which is motivated by the above Conjecture of Gross--Kuz'min and plays a crucial role in their descent formalism (see \cite[Thm.\@ 5.2]{BKS2}).

\begin{condition}[F] \label{condition-F}
The $\Z_p$-extension $K_\infty / K$ is such that the module of $\Gamma$-coinvariants $(A_\Sigma (K_\infty))_\Gamma$ is finite.
\end{condition}

\begin{rk} \label{gross-kuzmin-rk} The validity of condition (F) is known in the following important cases (see also \cite[\S\,2]{HoferKleine} for an overview of further results):
\begin{liste}
    \item If $k = \Q$, then Conjecture \ref{gross-kuzmin-conjecture} is (implicitly) proved by Greenberg \cite{Greenberg73}. 
    \item If there is exactly one prime of $K$ that ramifies in $K_\infty / K$, then condition (F) is a consequence of Chevalley's ambiguous class number formula (cf.\@ \cite[Ex.~2.7]{Kleine19}).
    \item If $|S_p (K) | \leq 2$ and $K_\infty$ is the cyclotomic $\Z_p$-extension of $K$, then the validity of Conjecture \ref{gross-kuzmin-conjecture} follows from (b) and a result of Kleine \cite[Thm.\@ B]{Kleine19}.\\
    Said result of Kleine hinges upon the following fact (see the proof of \cite[Lem.\@~3.5]{Kleine19}): Let $N / \Q$ be a normal extension and suppose that $x \in \bigO_{N, S_p (N)}$ is such that we have $\log_p \NN_{N_\p / \Q_p} (x) = 0$ for all $\p \in S_p (N)$. Then the valuation $\ord_\p (x)$ is the same for all $\p \in S_p (N)$. \\
However, the proof of this assertion given in \textit{loc.\@ cit.\@} contains an inaccuracy and we therefore take the opportunity to provide a better argument. Let $\p_0 \in S_p (N)$ be such that $n \coloneqq \ord_{\p_0} (x)$ is minimal among $\{ \ord_\p (x) \mid \p \in S_p (N) \}$. Write $e_{p}$ for the ramification degree of $N / \Q$ at $p$, then $x^{e_p} p^{- n}$ is a unit at $\p_0$ and integral at any other finite place of $N$. By assumption $\log_p \NN_{N_{\p_0} / \Q_p} (x) = 0$, hence also $\log_p \NN_{N_{\p_0} / \Q_p} (x^{e_p} p^{-n}) = 0$ and we can find an integer $m \geq 0$ such that $\NN_{N_{\p_0} / \Q_p} (x^{e_p} p^{-n})^m = 1$. Let $G_{\p_0} \subseteq \gal{N}{\Q}$ be the decomposition group at $\p_0$ and set $M = N^{G_{\p_0}}$. Then we have
\[
\NN_{N_{\p_0} / \Q_p} (x^{me_p} p^{-mn}) = \NN_{N / M} ( x^{me_p} p^{-mn}) = 1,
\]
so $x^{me_p} p^{-mn}$ is a unit in $\bigO_N$ and it follows that
\[
\ord_\p (x^{me_p} p^{-mn}) = 0 \iff \ord_\p (x) = n 
\]
for all $\p \in S_p (N)$. This finishes the proof of the claim. 
\item Suppose that condition (F) holds for a fixed $\Z_p$-extension $K_\infty$ of $K$. Kleine has proved \cite[Cor.~3.6]{Kleine17} that there exists an integer $n_0 \geq 1$ such that condition (F) also holds for all $\Z_p$-extensions $K_\infty'$ of $K$ with the following property: The $n$-th layers $K'_n$ and $K_n$ agree for all $n \geq n_0$, and $S_\ram (K_\infty' / K) \subseteq S_\ram (K_\infty / K)$. 
\end{liste}
\end{rk}

In \S\,\ref{condition-F-proof-section} we will prove condition (F) in new instances. 

\begin{lem} \label{Gross-Kuzmin-Lemma}
The following hold:
\begin{liste}
\item $( A_{\Sigma, T} (K_\infty) )_{\Gamma}$ is finite if and only if $( A_{\Sigma} (K_\infty) )_{\Gamma}$ is finite.
\item If $\Sigma'$ is a finite set of places of $k$ which contains $S_\infty (k) \cup S_\ram (K_\infty / K)$ and is such that $\Sigma' \subseteq \Sigma$, then $( A_{\Sigma, T} (K_\infty) )_{\Gamma}$ is finite as soon as $( A_{\Sigma', T} (K_\infty) )_{\Gamma}$ is. If no place in $\Sigma \setminus \Sigma'$ splits completely in $K_\infty / K$, then the converse is true as well.  
\end{liste}
\end{lem}

\begin{proof}
The exact sequence 
\begin{cdiagram}
    \mathbb{F}_{T_{K_n}}^\times : =   \bigoplus_{w \in T_{K_n}} \left ( 
       \faktor{\bigO_{K_n}}{w} \right )^\times \arrow{r} & 
       A_{\Sigma, T} (K_n) \arrow{r} & A_{\Sigma} (K_n) \arrow{r} & 0
\end{cdiagram}
implies that it is sufficient to show that the module $ \varprojlim_n \mathbb{F}_{T_{K_n}}^\times$ has finite $\Gamma_K$-coinvariants in order to prove (a). Now taking the limit of the exact sequences (which are obtained as representatives of the complexes $\bigoplus_{v \in T} \mathrm{R} \Gamma_f ( K_v, \mathrm{Ind}_{G_{K_n}}^{G_k} (\Z_p (1)) )$, see \cite[(19)]{BurnsFlach01})
\begin{equation} \label{T-modiciation-sequence-1}
\begin{tikzcd}
       0 \arrow{r} & \bigoplus_{v \in T} \Z_p [\cG_{K_n}] \arrow{rrr}{(1 - \text{N} v^{-1} \cdot \Frob_v)_v} & & & \bigoplus_{v \in T} \Z_p [\cG_{K_n}] \arrow{r} & 
        \mathbb{F}_{T_{K_n}}^\times \arrow{r} & 0
\end{tikzcd}
\end{equation}
gives an exact sequence
\begin{equation} \label{T-modiciation-sequence-2}
\begin{tikzcd}
       0 \arrow{r} & \bigoplus_{v \in T} \bLambda \arrow{r} & \bigoplus_{v \in T} \bLambda \arrow{r} & \varprojlim_n \mathbb{F}_{T_{K_n}}^\times \arrow{r} & 0.
\end{tikzcd}
\end{equation}
By taking $\Gamma$-coinvariants of (\ref{T-modiciation-sequence-2}) we obtain the exact sequence
\begin{cdiagram}
 \bigoplus_{v \in T} \Z_p [\cG] \arrow{rrr}{1 - \text{N} v^{-1} \cdot \Frob_v} & & & \bigoplus_{v \in T} \Z_p [\cG] \arrow{r} & 
       \Big ( \varprojlim_n \mathbb{F}_{T_{K_n}}^\times \Big)_\Gamma \arrow{r} & 0.
\end{cdiagram}
Comparing with (\ref{T-modiciation-sequence-1}), we deduce that $\big ( \varprojlim_n \mathbb{F}_{T_{K_n}}^\times \big)_{\Gamma} = \mathbb{F}_{T_{K}}^\times$. In particular, said module is finite.\smallskip\\
The first part of (b) is clear since $A_{\Sigma, T} (K_\infty)$ is a quotient of $A_{\Sigma', T} (K_\infty)$. For the second part we note that the assumption implies that any place $v \in \Sigma \setminus \Sigma'$ is inert in $K_m / K_n$ for big enough integers $n, m \geq 0$. It follows that the norm map $A_{\Sigma', T} (K_m) \to A_{\Sigma', T} (K_n)$ induces multiplication by $p^{m - n}$ on the class $[v]$. Thus, we must have $A_{\Sigma', T} (K_\infty) = A_{\Sigma, T} (K_\infty)$ and this proves the claim. 
\end{proof}

In the sequel, for any finite set $M$ of places of $k$, we set
\[
r_M (\chi) = \begin{cases}
| \{ v \in M \mid \chi ( \cG_v) = 1 \} | & \text{ if } \chi \neq 1, \\
|M| - 1 & \text{ if } \chi = 1.
\end{cases}
\]
Here $\cG_v \subseteq \cG$ denotes the decomposition group at $v$.\medskip\\ 
We are now in a position to state the main result of this subsection. 

\begin{thm} \label{Gross-Kuzmin-theorem}
For any character $\chi \in \widehat{\cG}$ such that $r_\Sigma (\chi) = r'$ the following assertions are equivalent.
\begin{enumerate}[label=(\roman*)]
\item The module $e_\chi \Q_p (\chi) A_{\Sigma, T} (K_\infty)^\Gamma$ vanishes. 
\item The map
\begin{align*}
\bigoplus_{v \in W} \Rec_{v, \chi} \: & e_\chi \Q_p (\chi) U_{K, \Sigma, T} \to e_\chi \Q_p (\chi)  Y_{K, W} \otimes_{\Z_p} \faktor{I ( \Gamma)}{I (\Gamma)^2 }, \\
& \quad a \mapsto e_\chi \sum_{v \in W} \sum_{\sigma \in \cG} \chi (\sigma) w \otimes \big ( \rec_w (\sigma^{-1} a) - 1   \big)
\end{align*}
is surjective.
\item The map
\[
\Rec_W \: e_\chi \Q_p (\chi) \exprod^{r'}_{\Z_p [\cG]} U_{K, \Sigma, T} \to e_\chi \Q_p (\chi) \exprod^r_{\Z_p [\cG]} U_{K, \Sigma, T}
\otimes_{\Z_p} \faktor{I ( \Gamma)^e}{I (\Gamma)^{e + 1} }
\]
defined in (\ref{Rec-map})
is non-zero (equivalently, injective).
\end{enumerate}
\end{thm}

The proof of this result will be given in \S\,\ref{bockstein-section}.

\begin{rk} \label{gross-kuzmin-rk-1}
If $K_\infty / K$ is the cyclotomic $\Z_p$-extension, then the equivalence (i) $\Leftrightarrow$ (ii) in Theorem~\ref{Gross-Kuzmin-theorem} is already known due to \cite[Thm.~1.14]{Kol91}. In general, the implication (i) $\Rightarrow$ (iii) is proved in \cite[\S\,5B]{BKS2}.\\ 
If $K$ is a CM extension of a totally real field $k$ and $\chi$ is totally odd, then (iii) is equivalent to the non-vanishing of the \textit{Gross regulator}. In this setting, Gross has proved in \cite[Prop.~1.16]{Gross} that condition (iii) holds if there is at most one prime $\p$ of $k$ above $p$ such that $\chi (\p) = 1$. 
\end{rk}

We end this subsection by recording the following technical observation that will prove useful in applications of Theorem \ref{Gross-Kuzmin-theorem}.

\begin{lem} \label{soerens-lemma}
Let $\chi \in \widehat{\cG}$ be a character and write $K_\chi$ for the subfield of $K$ cut out by the character $\chi$. The following assertions are equivalent:
\begin{enumerate}[label=(\roman*)]
\item $e_\chi \Q_p (\chi) A_{\Sigma, T} (K_\infty)_\Gamma = 0$,
\item $e_\chi \Q_p (\chi) A_{\Sigma, T} ( K_{\chi, \infty})_{\Gamma_\chi} = 0$. 
\end{enumerate}
Here $\Gamma_\chi = \gal{K_{\chi, \infty}}{K_\chi}$ denotes the Galois group of the $\Z_p$-extension $K_{\chi, \infty}$ of $K_\chi$. 
\end{lem}

\begin{proof}
Let $m$ be such that $K \cap K_{\chi, \infty} = K_{\chi, m}$ and write $H$ for $\gal{K}{K_{\chi, m}}$, which we can identify with $\gal{K_\infty}{K_{\chi, \infty}}$ and therefore view as a subgroup of $\gal{K_{\chi, \infty}}{k}$. Observe that the norm maps $\NN_{K_n / K_{\chi, n + m}} \: A_{\Sigma, T} (K_n) \to A_{\Sigma, T} ( K_{\chi, n + m})$ induce a map 
\[
\NN_{K_\infty / K_{\chi, \infty}} \: A_{\Sigma, T} (K_\infty) \to A_{\Sigma, T} (K_{\chi, \infty})
\]
which factors as 
\begin{cdiagram}
A_{\Sigma, T} (K_\infty) \arrow{r}{\cdot \NN_H} \arrow{d}[left]{\NN_{K_\infty / K_{\chi, \infty}}} & A_{\Sigma, T} (K_\infty)^H \\
A_{\Sigma, T} (K_{\chi, \infty}) \arrow{ur}[right, yshift=-0.2cm]{i} & 
\end{cdiagram}
where $i$ is the natural map induced by the inclusions $K_{\chi, n} \subseteq K_n$. 
Define a height-one prime ideal of $\bLambda$ as $\p = \ker \{ \bLambda \stackrel{\chi}{\longrightarrow} \Z_p [\im \chi] \}$.\\
We have $\chi (\NN_H) = |H| \neq 0$, hence $\NN_H \in \bLambda_\p^\times$ and so multiplication by $\NN_H$ (which is the same as $i \circ \NN_{K_\infty / K_{\chi, \infty}}$) becomes an isomorphism. Moreover, the composite $\NN_{K_\infty / K_{\chi, \infty}} \circ i$ coincides with multiplication by $|H|$ and is therefore also bijective after localisation at $\p$. It follows that $\NN_{K_\infty / K_{\chi, \infty}}$ induces an isomorphism 
\[
\big ( A_{\Sigma, T} (K_\infty) \big)_\p \stackrel{\simeq}{\longrightarrow} \big ( A_{\Sigma, T} (K_{\chi, \infty}) \big)_\p.
\]
Now, we have an isomorphism of $\bLambda$-modules $\faktor{\bLambda_\p}{\p \bLambda_\p} \cong e_\chi \Q_p (\chi)$ and thus obtain 
 \begin{align*}
e_\chi \Q_p(\chi) \otimes_{\Z_p [\cG]} A_{\Sigma, T} (K_\infty)_\Gamma & \cong  \faktor{\bLambda_\p}{\p \bLambda_\p} \otimes_\bLambda A_{\Sigma, T} (K_\infty) \\
& \cong \faktor{\bLambda_\p}{\p \bLambda_\p} \otimes_\bLambda A_{\Sigma, T} (K_{\chi, \infty}) \\
& \cong \faktor{\bLambda_\p}{\p \bLambda_\p} \otimes_{\Z_p [\cG_\chi]} \big ( \Z_p [\cG_\chi] \otimes_\bLambda A_{\Sigma, T} (K_{\chi, \infty}) \big) \\
& \cong e_\chi \Q_p(\chi) \otimes_{\Z_p [\cG_\chi]} A_{\Sigma, T} (K_{\chi, \infty})_{\Gamma_\chi},
 \end{align*}
 thereby proving the Lemma. 
 \end{proof}

\subsection{Computation of Bockstein homomorphisms} \label{bockstein-section}

To prove Theorem \ref{Gross-Kuzmin-theorem} we will perform a computation of \textit{Bockstein maps} as in \cite[\S\,5B]{BKS2} (see also \cite[\S\,10]{Burns07}). 
The Bockstein homomorphism at level $n$ is defined as the map
\[
\beta_n \: U_{K, \Sigma, T} \to H^1 (D^\bullet_{K_n, \Sigma, T}) \otimes_{\Z_p} I ({\Gamma_n}) 
\to H^1 (D^\bullet_{K, \Sigma, T}) \otimes_{\Z_p} \faktor{ I ( \Gamma_n)}{I ({\Gamma_n})^2}, 
\]
where the first arrow is the connecting homomorphism arising from the exact triangle
\begin{cdiagram}
D^\bullet_{K_n, \Sigma, T} \otimes_{\Z_p [\cG_n]} I_{\Gamma_n} \arrow{r} & 
D^\bullet_{K_n, \Sigma, T} \arrow{r} & D^\bullet_{K_n, \Sigma, T} \otimes_{\Z_p [\cG_n]} \Z_p [\cG]
\arrow{r} & \phantom{X}.
\end{cdiagram}
By taking the limit over $n$, we obtain a map 
\begin{equation} \label{bockstein-Iwasawa-1}
\beta_\infty \: U_{K, \Sigma, T} \to H^1 (D^\bullet_{K, \Sigma, T}) \otimes_{\Z_p} \varprojlim_n \faktor{ I ( \Gamma_n)}{I ({\Gamma_n})^2} \cong H^1 (D^\bullet_{K, \Sigma, T})  \otimes_{\Z_p} \faktor{ I (\Gamma)}{I (\Gamma)^2 } 
\end{equation}
that is identified with the map
\[
U_{K, \Sigma, T} \stackrel{\delta'}{\to} H^1 (D^\bullet_{K_\infty, \Sigma, T} ) \otimes_{\Z_p} I (\Gamma) \to 
 H^1 (D^\bullet_{K, \Sigma, T})  \otimes_{\Z_p} \faktor{ I (\Gamma)}{I (\Gamma)^2 } 
,\]
where the map $\delta'$ is the connecting homomorphism induced by the triangle
\begin{cdiagram}
D^\bullet_{K_\infty, \Sigma, T} \otimes_{\bLambda} I_{\Gamma} \arrow{r} & 
D^\bullet_{K_\infty, \Sigma, T} \arrow{r} & D^\bullet_{K_\infty, \Sigma, T} \otimes_{\bLambda} \Z_p [\cG]
\arrow{r} & \phantom{X}.
\end{cdiagram}
To make the definition of $\beta_\infty$ more explicit we now fix a topological generator $\gamma \in \Gamma$ and note that this choice
gives rise to an identification of the above triangle with the triangle induced by multiplication by $\gamma - 1$. In particular, it allows to view $\delta'$ as the boundary homomorphism $\delta$ arising from an application of the snake lemma to the following commutative diagram:
\begin{cdiagram}[row sep=small]
       & & & 0 \arrow{d} & \\
       & & & H^0 (D^\bullet_{K_n, \Sigma, T}) \arrow{d} & \\
       0 \arrow{r} & \Pi_\infty \arrow{d}{\phi} \arrow{r}{\cdot (\gamma - 1)} &
       \Pi_\infty \arrow{d}{\phi} \arrow{r} & \Pi_0 \arrow{r} \arrow{d}{\phi_0} & 0 \\
         0 \arrow{r} & \Pi_\infty \arrow{d} \arrow{r}{\cdot (\gamma - 1)} &
       \Pi_\infty \arrow{r} & \Pi_0 \arrow{r}  & 0 \\
       & H^1 (D^\bullet_{K_\infty, \Sigma, T}) \arrow{d} & & & \\
       & 0 & & &
\end{cdiagram}
Using that $(U_{K_\infty, \Sigma, T})_\Gamma \cong \UN^1_0$, the snake lemma also shows that this boundary map $\delta$ fits into the exact sequence
\begin{equation} \label{norms-sequence}
\begin{tikzcd}
     0 \arrow{r} & \UN^1_0 \arrow{r} & U_{K, \Sigma, T} \arrow{r}{\delta} & H^1 ( D^\bullet_{K_\infty, \Sigma, T})^\Gamma \arrow{r} & 0.
     \end{tikzcd}
\end{equation}
Our fixed choice of topological generator $\gamma \in \Gamma$ also induces an isomorphism $\faktor{I (\Gamma)}{I (\Gamma)^2} \cong \Z_p$, hence $\beta_\infty$ can be identified with the composite map
\begin{equation} \label{composite-bockstein}
U_{K, \Sigma, T} \stackrel{\delta}{\longrightarrow} H^1 (D^\bullet_{K_\infty, \Sigma, T})^\Gamma \subseteq H^1 (D^\bullet_{K_\infty, \Sigma, T}) \to H^1 (D^\bullet_{K_\infty, \Sigma, T})_\Gamma \cong H^1 (D^\bullet_{K, \Sigma, T}).
\end{equation}

\begin{lem} \label{bockstein-lem-1}
Let $v \in \Sigma$ and recall that we have previously fixed a place $w$ of $K$ lying above $v$. 
The composite map 
\begin{align*}
\beta_w \: U_{K, \Sigma, T}  & \stackrel{\beta_\infty}{\longrightarrow} H^1 (D^\bullet_{K, \Sigma, T})  \otimes_{\Z_p} \faktor{ I (\Gamma)}{I (\Gamma)^2 }
\stackrel{\pi_K}{\to} X_{K, \Sigma } \otimes_{\Z_p} \faktor{ I (\Gamma)}{I (\Gamma)^2 } \\
&
  \stackrel{w^\ast}{\longrightarrow} \Z_p [\cG] \otimes_{\Z_p} \faktor{ I (\Gamma)}{I (\Gamma)^2 } 
\end{align*}
is zero if $v  \in V$, and coincides with $\Rec_v$ if $v \in W$. Here we have used the map $\pi_K$ appearing in (\ref{yoneda-finite}) and the notation $w^\ast$ for the $\Z_p [\cG]$-linear dual of $w$ considered as an element of $Y_{K, \Sigma}$.
\end{lem}

\begin{proof}
This follows immediately from the corresponding results on $\beta_n$, see \cite[Lem.~5.20 and Lem.~5.21]{BKS}. 
\end{proof}

\textit{Proof (of Theorem \ref{Gross-Kuzmin-theorem}):}
The assumption $r_\Sigma (\chi) = r'$ implies that, if $\chi$ is non-trivial, we have $\chi (v_0) \neq 1$. We therefore have an isomorphism
 \[
 e_\chi \Q_p Y_{K, V'} \to e_\chi \Q_p X_{K, \Sigma}, \quad e_\chi \cdot \sum_{v \in V'} a_v w \mapsto e_\chi \cdot \sum_{v \in V'} a_v (w  - w_0). 
 \]
 Write $\alpha$ for the natural map $H^1 (D^\bullet_{K_\infty, \Sigma, T}) \to H^1 (D^\bullet_{K_\infty, \Sigma, T})_\Gamma \cong H^1 (D^\bullet_{K, \Sigma, T})$. 
By (\ref{composite-bockstein}) and Lemma \ref{bockstein-lem-1}, 
the map in (ii) then coincides with the composite
\begin{align*}
e_\chi \beta_\infty \:   e_\chi \Q_p (\chi) U_{K, \Sigma, T} & \stackrel{\delta}{\longrightarrow} 
e_\chi \Q_p (\chi) H^1 (D^\bullet_{K_\infty, \Sigma, T})^\Gamma \\
& \stackrel{\alpha}{\longrightarrow} e_\chi \Q_p (\chi) H^1 (D^\bullet_{K, \Sigma, T}) \\
& \stackrel{\pi_K}{\cong} e_\chi \Q_p (\chi) Y_{K, V'}
,
\end{align*}
and actually has image inside $e_\chi \Q_p (\chi) Y_{K, W}$. The first map $\delta$ is already surjective in any case by the exact sequence (\ref{norms-sequence}), so the above composite map surjects onto $e_\chi \Q_p (\chi) Y_{K, W}$ if and only if $\alpha$ does.\\
Now, we have a commutative diagram
\begin{equation} \label{some-diagram}
\begin{tikzcd}
e_\chi \Q_p (\chi) H^1 (D^\bullet_{K_\infty, \Sigma, T} )^\Gamma \arrow{r}{\alpha} \arrow{d}{\pi_{K, \infty}} & 
e_\chi \Q_p (\chi) H^1 (D^\bullet_{K, \Sigma, T}) \arrow{d}[right]{\pi_K}[left]{\simeq} \\
e_\chi \Q_p (\chi) Y_{K_\infty, V'}^\Gamma \arrow{r}{f} & 
e_\chi \Q_p (\chi) Y_{K, V'} \\
\end{tikzcd}
\end{equation}
where the map $f$ is induced by \begin{equation} \label{map-on-Y}
Y_{K_\infty, V'}^\Gamma \to Y_{K, V'}, \quad \sum_{i = 1}^{r'}  (a_{n, i} \cdot w_{K_n, i})_n \mapsto \sum_{i = 1}^{r'} a_{0, i} \cdot w_{K, i}
\end{equation}
via extension of scalars.
Observe that $Y_{K_\infty, V}$ is a $\Z_p \llbracket \Gamma \rrbracket$-projective direct summand of $Y_{K_\infty, V'}$. Given this, we have $Y_{K_\infty, V'}^\Gamma = Y_{K_\infty, W}^\Gamma$. We now claim that the map in (\ref{map-on-Y}) embeds $Y_{K_\infty, W}^\Gamma$ with finite index into $Y_{K, W}$.\\
To do this, we write $l_i$ for the index 
of the decomposition group of $w_i$ inside $\Gamma$.
For $1\leq i \leq r'$, each $w_i$ splits completely in $K / k$ and hence $K_{l_i}$ coincides with the decomposition field of $w_i$ inside the extension $K_\infty / k$. 
It follows that
\begin{align*}
Y_{K_\infty, W}^\Gamma & \cong \big ( \bigoplus_{i = r + 1}^{r'} \Z_p [\cG_{l_i}] \big)^\Gamma = \bigoplus_{i = r + 1}^{r'}  \Z_p [\cG_{l_i}]^\Gamma
=  \bigoplus_{i = r + 1}^{r'} \Z_p [\cG_{l_i}]^{\Gamma_{l_i}}
\end{align*}
As a consequence, the map (\ref{map-on-Y}) is given by
\[
Y_{K_\infty, W}^\Gamma \cong \bigoplus_{i = r + 1}^{r'} \Z_p [\cG_{l_i}]^{\Gamma_{l_i}} 
\longrightarrow 
\bigoplus_{i = r + 1}^{r'} \Z_p [\cG_{l_i}]_{\Gamma_{l_i}}
\cong 
\bigoplus_{i = r + 1}^{r'} \Z_p [\cG] \cong 
Y_{K, W}, 
\]
where the middle arrow is the natural projection map. Since, for each $i$, the cokernel of the map $\Z_p [\cG_{l_i}]^{\Gamma_{l_i}} \to \Z_p [\cG_{l_i}]_{\Gamma_{l_i}}$ is annihilated by $p^{l_i} = |\Gamma_{l_i}|$, this shows the claim. It follows that the map $f$ in (\ref{some-diagram}) maps $e_\chi \Q_p (\chi) Y_{K_\infty, W}^\Gamma$ isomorphically onto $e_\chi \Q_p (\chi) Y_{K, W}^\Gamma$ and, because the image of $\pi_K \circ \alpha$ is contained in $e_\chi \Q_p (\chi) Y_{K, W}$, we are therefore reduced to the question of when the map labelled $\pi_{K, \infty}$ in the diagram (\ref{some-diagram}) is surjective. From the exact sequence (\ref{yoneda-Iwasawa}) we obtain the exact sequence 
\begin{equation} \label{invariants-H2-sequence}
    \begin{tikzcd}
       0 \arrow{r} & (A_{\Sigma, T} (K_\infty))^\Gamma \arrow{r} & H^1 (D^\bullet_{K_\infty, \Sigma, T})^\Gamma \arrow{r}{\pi_{K, \infty}} & (X_{K_\infty, \Sigma})^\Gamma, 
       \end{tikzcd}
\end{equation}
 hence $\pi_{K, \infty}$ is surjective (in fact, an isomorphism) after extending scalars to $e_\chi \Q_p (\chi)$ if and only if $e_\chi \Q_p (\chi) (A_{\Sigma, T} (K_\infty))^\Gamma = 0$. This establishes (i) $ \Leftrightarrow$ (ii). \medskip \\
 As note above, we can identify the map $e_\chi \beta_\infty$ with the map $\bigoplus_{v \in W} \Rec_{v, \chi}$ that appears in statement (ii) of Theorem \ref{Gross-Kuzmin-theorem}. 
By construction $\ker \beta_\infty$ contains $\UN^1_0$. Thus, $e_\chi \Q_p (\chi) \ker \beta_\infty$ has dimension at least $r$ by Lemma \ref{universal-norms-p-units}\,(c). As a consequence, the exterior power $e_\chi \Q_p (\chi) \exprod^r_{\Z_p [\cG]} \ker \beta_\infty$ is non-zero and so the equivalence (ii) $\Leftrightarrow$ (iii) follows upon appealing to \cite[Lem.~4.2]{BKS}. \qed

\subsection{Proof of condition (F) in special cases} \label{condition-F-proof-section}

In this section we shall explain how one can prove the equivalent conditions of Theorem \ref{Gross-Kuzmin-theorem} in special cases. Crucial ingredient in these arguments is the following Lemma, which is a direct consequence of Brumer's $p$-adic analogue of Baker's Theorem from transcendence theory. 

\begin{lem} \label{brumer-baker-lem}
Let $v$ be a place of $k$ that splits completely in $K$. Then there is an element $a \in \bigO_{K, \{ v \}}^\times$ such that 
\[
\sum_{\sigma \in \cG} \chi (\sigma) \cdot \log_p ( \iota_w (\sigma^{-1} a)) \neq 0
\]
for all non-trivial characters $\chi \in \widehat{\cG}$.
\end{lem}

\begin{proof}
Let $a \in \bigO_{K, \Sigma}^\times$ be an element that is only divisible by $w$ and no other finite prime of $K$. For example, such an element is given by any generator of $w^{h_K}$, where $h_K$ is the class number of $K$. Fix a non-trivial character $\chi \in \widehat{\cG}$ and suppose that
\begin{align} \label{log-vanishes-1}
 \sum_{\sigma \in \cG}  \chi (\sigma) \cdot \log_p ( \iota_w (\sigma^{-1} a)) = 0.
\end{align}
The elements $\chi (\sigma)$ are algebraic over $\Q$ and not zero, hence Brumer's $p$-adic analogue of Baker's theorem \cite{Brumer} (see also \cite[Thm.\@ 10.3.14]{NSW}) asserts the existence of integers $n_\sigma \in \Z$, not all of them zero, such that
\[
\log_p  \big ( \iota_w \big (\sum_{\sigma \in \cG} n_\sigma \sigma a \big ) \big ) = 0. 
\]
After multiplying by a suitable integer if necessary we may therefore assume that $\sum_{\sigma \in \cG} n_\sigma \sigma a$ is a power of $p$. In fact, this power of $p$ needs to be $p^{n_1 \ord_w (a) / e_{w | p}}$ by choice of $a$. Here $e_{w | p}$ denotes the ramification degree of $w$ in $K / \Q$. We have $e_{w | p} = e_{\sigma w | p}$ for all $\sigma \in \cG$, so it follows that
\begin{align*}
n_1 \ord_w (a) & = 
\ord_w p^{n_1 \ord_w (a) / e_{w | p}} =
\ord_{\sigma w} p^{n_1 \ord_w (a) / e_{w | p}} = \ord_{\sigma w} \big ( \sum_{\sigma \in \cG} n_\sigma \sigma a \big) \\
& =  n_\sigma \ord_w (a)
\end{align*}
for all $\sigma \in \cG$. We deduce that all $n_\sigma$ agree, and so $\log_p  ( \iota_w (\sum_{\sigma \in \cG} \sigma a)) = 0$ as well. This combines with (\ref{log-vanishes-1}) to imply that
\[
 \sum_{\sigma \in \cG \setminus \{ 1 \}}  (\chi (\sigma) - 1) \cdot \log_p ( \iota_w (\sigma^{-1} a)) = 0.
\]
By assumption, $\chi \neq 1$ and so we can apply the Theorem of Brumer-Baker yet again to obtain integers $m_\sigma$, not all of them zero, such that $x = \iota_w \big ( \sum_{\sigma \in \cG \setminus \{ 1 \}} m_\sigma \sigma a \big )$ lies 
in the kernel of the $p$-adic logarithm.
Now, $x$ is integral at $w$ and therefore we must have that $x$ is a root of unity. However, the set $\{ \sigma a \mid \sigma \in \cG \}$ is $\Z$-linearly independent and so this can only happen if all $m_\sigma$ are zero, which is a contradiction.
\end{proof}

\begin{thm} \label{new-gross-kuzmin-thm-cyclotomic}
Assume the following conditions:
\begin{enumerate}[label=(\roman*)]
    \item $p$ splits completely in $k / \Q$,
    \item if $p = 2$, then all infinite places split completely in $K_\infty / k$,
    \item for each non-trivial character $\chi \in \widehat{\cG}$ there is at most one finite place $v \in \Sigma$ which ramifies in $k_\infty / k$ and is such that $\chi (\cG_v) = 1$,
    \item either $k / \Q$ is abelian or $|S_\ram (k_\infty / k)| \leq 2$.
\end{enumerate}
Then $A_{\Sigma, T} (K_\infty)^\Gamma$ is finite. 
\end{thm}

\begin{rk} \label{new-gross-kuzmin-thm-rk}
\begin{liste}
    \item The conclusion of Theorem \ref{new-gross-kuzmin-thm-cyclotomic} remains true if instead of condition (iv) one assumes that $A_{\Sigma, T} (k_\infty)^{\gal{k_\infty}{k}}$ is finite. This corresponds with Theorem \ref{thm-C}\,(a) in the introduction. 
    \item We give two examples of concrete situations in which Theorem \ref{new-gross-kuzmin-thm-cyclotomic} can be applied.
    \begin{enumerate}[label=(\roman*)]
        \item Suppose that $k$ is an imaginary quadratic field in which $p$ splits completely. If we fix a prime ideal $\p$ of $k$ above $p$, then there is a unique $\Z_p$-extension $k_\infty$ of $k$ that is unramified outside $\p$. Given this, Theorem~\ref{new-gross-kuzmin-thm-cyclotomic} implies that condition (F) holds for all abelian extensions $K / k$ with respect to the $\Z_p$-extension $K_\infty = K \cdot k_\infty$ of $K$. \\
        We remark, however, that this fact is already known, see the proof of \cite[Thm.\@~1.4]{Rub88} where it is deduced from the known validity of Leopoldt's Conjecture in this setting (the latter is of course also derived from the Theorem of Brumer-Baker, so our proof can be considered more direct). 
        \item Suppose that $K = k$ is a CM field but not imaginary quadratic. Assume that $p$ splits completely in $k / \Q$ and fix a prime $\p$ of $k$ lying above $p$. We denote by $\overline{\p}$ the complex conjugate of $\p$. By class field theory, there exists a $\Z_p$-extension $k_\infty$ of $k$ that is unramified outside $\{ \p, \overline{\p} \}$. Theorem \ref{new-gross-kuzmin-thm-cyclotomic} now implies that $A_{\Sigma, T} (k_\infty)^\Gamma$ is finite. 
    \end{enumerate}
\end{liste}
\end{rk}

\begin{proof}
Let $\chi$ be a character of $\cG$. By Lemma \ref{soerens-lemma} we may assume that $K$ is the field cut out by the character $\chi$.
Due to Lemma \ref{Gross-Kuzmin-Lemma}\,(b) we may moreover assume that $S = S_\infty (k)$, i.e.\@ $\Sigma = S_\infty (k) \cup S_\ram (K_\infty  / k)$. \\
Let us first consider the case that $\chi \neq 1$ is non-trivial. Take $V' = \{ v \in \Sigma \mid \chi (v) = 1 \}$, then assumption (ii) ensures that $W = V' \setminus V$ only contains finite places. It follows that $W$ must be contained in $S_\ram (K_\infty / K)$. If $W$ is empty, there is nothing to prove, so we may assume that $W = \{ v \}$ for a single place $v \in S_p (k)$. Given this, we have $r_\Sigma (\chi) = r + 1$ and so may apply Theorem \ref{Gross-Kuzmin-theorem}. We shall now show that statement (ii) in \ref{Gross-Kuzmin-theorem} holds true in this situation. \\
The codomain of $\Rec_{v, \chi}$ is of $\Q_p (\chi)$-dimension one, hence the map $\Rec_{v, \chi}$ is surjective as soon as it is non-zero.\\  
Let $\Gamma_w \subseteq \Gamma$ be the decomposition group at $w$ and write $d$ for the index $(\Gamma_w : I_w)$ of the inertia group $I_w$ at $w$.
Let $H'$ be the unique unramified extension of $\Q_p$ of degree $d$ and
write $H'_\infty$ for the maximal extension of $H'$ that is totally ramified and abelian over $\Q_p$. This extension can be explicitly described using relative Lubin-Tate theory. In particular, we have an isomorphism
\begin{equation} \label{artin-map-1}
1 + p \Z_p \stackrel{\simeq}{\longrightarrow} \gal{H'_\infty}{H}
\end{equation}
that coincides with the composition of the inclusion $\Q_p^\times \hookrightarrow (H')^\times$ and the local reciprocity map $(H')^\times \to \gal{H'_\infty}{H}$, see \cite[Ch.\@ I, Prop.\@ 1.8]{deS87}. We deduce that $\gal{H'_\infty}{H'}$ splits as the direct sum of $I_w$ and a finite part. Thus,  we have an isomorphism
\[
\Gamma_w^d = I_w \stackrel{\simeq}{\longrightarrow} \Z_p, \quad \sigma \mapsto \log_p \chi_\text{ell} ( \sigma),
\]
where $\chi_\text{ell}$ denotes the inverse map of (\ref{artin-map-1}). We can therefore identify the map $d \Rec_{v, \chi}$ with the map
\begin{align*}
  e_\chi \Q_p (\chi) U_{K, \Sigma, T} \longrightarrow  e_\chi \Q_p (\chi) Y_{K, W} \otimes_{\Z_p} \Gamma_w^d \cong e_\chi \Q_p (\chi),  
 \quad a \mapsto  - de \cdot  \sum_{\sigma \in \cG} \chi (\sigma) \cdot \log_p  ( \iota_{ w} (\sigma^{-1} a)).
\end{align*}
Now, Lemma \ref{brumer-baker-lem} implies that this map is non-zero, as desired. \\
Let us finally assume that $\chi = 1$ is the trivial character. In this case we may assume that $K = k$ (by Lemma \ref{soerens-lemma}) and it is sufficient to show that
$A_{\Sigma, T} (k_\infty)^\Gamma$ is finite. 
If $k / \Q$ is abelian, then this holds true by a result of Greenberg (see Remark \ref{gross-kuzmin-rk}\,(a)), and if $|S_\ram (k_\infty / k)| = 1$, then this is covered by Remark \ref{gross-kuzmin-rk}\,(b). It remains to investigate the case $|S_\ram (k_\infty / k)| = 2$. Fix a place $v_0 \in S_\ram (k_\infty / k)$ and set $V' = \Sigma \setminus \{ v_0 \}$. In this situation we have $W = \{ v \}$ for a single place $v \in S_p (k)$ and, in particular, $r_\Sigma (1) = |V'|$. We may therefore apply Theorem \ref{Gross-Kuzmin-theorem} and, by the discussion above, it is enough to prove that the map
\[
\Q_p U_{K, \Sigma, T} \to \Q_p, \quad a \mapsto - de \log_p (\iota_w (a))
\]
is non-zero. This is however clear because the kernel of the $p$-adic logarithm is $\mu_{p - 1} \cdot p^\Z$, hence does not contain any element of $\Q U_{K, \Sigma, T}$ which is integral at $p$. 
\end{proof}

If $p$ does not split completely in $k$, the situation is much more complicated. We are however able to prove the following result concerning the case of $k$ being an imaginary quadratic field.

\begin{thm} \label{exists-Zp-extension}
Assume that $k$ is an imaginary quadratic field such that $p$ does not split in $k / \Q$. There are infinitely many $\Z_p$-extensions $k_\infty$ of $k$ such that all of the following conditions are satisfied:
\begin{liste}
    \item $A_{\Sigma, T} (K_\infty)^\Gamma$ is finite,
    \item at most two finite places of $k$ split completely in $k_\infty / k$, neither of them contained in $S (K) \cup S_p (k)$.
\end{liste}
\end{thm}

\begin{proof}
By Lemma \ref{soerens-lemma} the property (a) is satisfied if, for every character $\chi \in \widehat{\cG}$, the module $e_\chi \Q_p (\chi) A_{\Sigma, T} ( K_{\chi, \infty})$ vanishes, where $K_{\chi, \infty} = K_\chi \cdot k_\infty$ and $K_\chi$ denotes the subfield of $K$ cut out by the character $\chi$. 
By Remark \ref{gross-kuzmin-rk}\,(b) this holds for $\chi = 1$ because $k$ contains only prime above $p$, so it suffices to consider non-trivial characters. By Lemma \ref{Gross-Kuzmin-Lemma} it is enough to check if $e_\chi \Q_p (\chi) A_{\Sigma_\chi, T} ( K_{\chi, \infty})$ vanishes, where $\Sigma_\chi = S_\ram (K_{\chi, \infty} / k) \cup S_\infty (k)$. In this situation we may apply Theorem \ref{Gross-Kuzmin-theorem} which asserts that the aforementioned vanishing is equivalent to the surjectivity of the map $\bigoplus_{v \in W_\chi} \Rec_{v, \chi}$ defined in (ii) of Theorem \ref{Gross-Kuzmin-theorem} as
\[
 e_\chi \Q_p (\chi) \bigO_{K_\chi, \Sigma_\chi, T}^\times \to e_\chi \Q_p (\chi) Y_{K_\chi, W_\chi} \otimes_{\Z_p} \Gamma_\chi,
\quad a \mapsto \sum_{w \in W_\chi}\sum_{\sigma \in \cG_\chi} \chi (\sigma) w \otimes (\rec_w ( \sigma^{-1} a ) - 1),
\]
where $\cG_\chi = \gal{K_\chi}{k}$, $\Gamma_\chi = \gal{K_{\chi, \infty}}{K_\chi}$, and $W_\chi = \{v \in \Sigma_\chi \setminus S_\infty (k) \mid \chi (v) = 1 \}$.
Observe that we must have $W_\chi \subseteq S_\ram (k_\infty / k) = \{ v\}$ for the unique place $v$ of $k$ above $p$. If $W_\chi = \emptyset$, there is nothing to show. We may therefore assume that $W_\chi = \{ v \}$, and we let $\widehat{\cG}_W$ be the subset of $\widehat{\cG}$ comprising all non-trivial characters $\chi$ such that $W_\chi \neq \emptyset$. \\
Note that we have a commutative diagram
\begin{cdiagram}
 e_\chi \Q_p (\chi) \bigO_{K, \Sigma_\chi, T}^\times \arrow{r}{e_\chi \Rec_v} \arrow{d}[left]{\NN_{K / K_\chi}} & e_\chi \Q_p (\chi) \otimes_{\Z_p} \Gamma \arrow{d}{\simeq} \\
 e_\chi \Q_p (\chi) \bigO_{K_\chi, \Sigma_\chi, T} \arrow{r}{\Rec_{v, \chi}} &
e_\chi \Q_p (\chi) \otimes_{\Z_p} \Gamma_\chi,
\end{cdiagram}
where the isomorphism on the right is induced by the inclusion $\Gamma = \gal{K_\infty}{K} \subseteq \Gamma_\chi$ (which has finite index). It is therefore sufficient to check if the map $e_\chi \Rec_v$ is surjective (or, equivalently, non-zero) for all $\chi \in \widehat{\cG}_W$. \\
The basic strategy of the remainder of this proof is now to show that this holds if one avoids, if necessary, certain 'bad' $\Z_p$-extensions. As a first step towards this, we will now first give a more explicit description of the map $e_\chi \Rec_v$. \medskip \\
Let $F_\infty$ be the compositum of all $\Z_p$-extensions of $k$, which is a $\Z_p^2$-extension as a consequence of the known validity of Leopoldt's Conjecture for this setting. In fact, we know that $\gal{F_\infty}{k} = \Z_p \gamma_\cyc \oplus \Z_p \gamma_\anti$, where $\gamma_\cyc, \gamma_\anti \in \gal{F_\infty}{k}$ are such that the fixed fields $F_\infty^{\langle \gamma_\cyc \rangle}$ and $F_\infty^{\langle \gamma_{\anti}\rangle}$ are the cyclotomic and anti-cyclotomic $\Z_p$-extensions of $k$, respectively. 
Write $\gal{F_\infty}{k}_v = \gal{F_\infty K}{K}_w$ for a choice of decomposition group at $v$ inside $F_\infty / k$ and $w$ inside $\gal{F_\infty K}{K}$, respectively. If $I_v$ denotes the inertia subgroup of $\gal{F_\infty}{k}_v$, then explicit local class field theory \cite[Ch.\@ I, Prop.\@ 1.8]{deS87} gives that the inverse of the local reciprocity map identifies $I_v$ with a quotient of $1 + \p_v$, where $\p_v$ is the maximal ideal of the valuation ring $\bigO_{k_v} \subseteq k_v$  of $k_v$. Since $1 + \p_v$ and $I_v$ are both of $\Z_p$-rank two, $I_v$ must agree with the torsion-free part of $1 + \p_v$. 
We can therefore find an integer $s \geq 1$ such that
\begin{equation} \label{some-map}
\gal{F_\infty}{k}_v^{p^sd} \subseteq I_v^{p^s} \stackrel{\simeq}{\longrightarrow} (1 + \p_v)^{p^s} \subseteq 1 + \p_v^s \stackrel{\simeq}{\longrightarrow} \p_v^s,
\end{equation}
where $d = (\gal{F_\infty}{k} : I_v)$, the first arrow is the inverse of the local reciprocity map $\mathrm{Art}_v \: k_v^\times  \to \gal{F_\infty}{k}_v$, and the second arrow is the $p$-adic logarithm. 
The cokernel of (\ref{some-map}) is finite, hence, for all characters $\chi \in \widehat{\cG}_W$, it induces an isomorphism
\[
\omega_\chi \: e_\chi \Q_p (\chi) \otimes_{\Z_p} \gal{F_\infty \cdot K}{K} = e_\chi \Q_p (\chi) \otimes_{\Z_p} \gal{F_\infty}{k}  \stackrel{\simeq}{\longrightarrow} e_\chi \Q_p (\chi) \otimes_{\Z_p} \p_v.
\]
Given this, we can identify the map
\[
e_\chi \Q_p (\chi) \cdot U_{K, W, T} \to e_\chi \Q_p (\chi) \otimes_{\Z_p} \gal{F_\infty K}{K}, \quad
a \mapsto p^sd \sum_{\sigma \in \cG} \chi (\sigma) \otimes  \mathrm{Art}_v ( \iota_w (\sigma^{-1} a)),
\]
where $\iota_w \: K^\times \hookrightarrow K_w^\times$ denotes the canonical embedding,
with the map 
\begin{align*}
 \widetilde{\rho_\chi} \: e_\chi \Q_p (\chi) \otimes_\Z \bigO_{K, W, T}^\times \to e_\chi \Q_p (\chi) \otimes_{\Z_p} \p_v, \quad
a \mapsto - p^s d \sum_{\sigma \in \cG} \chi (\sigma)  \otimes \log_p ( \iota_w (\sigma^{-1} a)) .
\end{align*}
Let $\gamma, \delta \in \gal{F_\infty}{k}$ be a $\Z_p$-basis
and write $k_\delta = F_\infty^{\langle \delta \rangle}$ for the $\Z_p$-extension of $k$ that is cut out by $\delta$. Observe that all $\Z_p$-extensions of $k$ are of this form. We also set $K_\delta = K \cdot k_\delta$.\\
Recall that the map $\rec_w$ is the composite of $\iota_w \: K^\times \hookrightarrow K_w^\times$ and the local reciprocity map $K_w^\times \to \gal{K_\delta}{K}$, and note that the latter map can be described as the composition of $\mathrm{Art}_v$ and the restriction map on decomposition groups $\gal{F_\infty K}{K}_w \to \gal{K_\delta}{K}_w$. More explicitly, if
$\gamma^x \delta^y$ is an element of $\gal{F_\infty K}{K}_w$, then its restriction to $K_\delta$ coincides with $\gamma^x$. \\
Observe that it is sufficient to check the non-vanishing of the map $e_\chi \Rec_v$ after multiplication by $p^s d$.
The above discussion implies that $p^s d\cdot e_\chi \Rec_v$ can be identified with the map $\rho_\chi = \pi_{\chi, \gamma} \circ \widetilde{\rho_\chi}$, where $\pi_{\chi, \gamma}$ denotes the projection map
\[
\pi_{\chi, \gamma} \: e_\chi \Q_p (\chi) \otimes_{\Z_p} \p_v \to e_\chi \Q_p (\chi), \quad
x \omega_\chi (\gamma^{p^s d}) + y \omega_\chi ( \delta^{p^s d}) \mapsto x.
\]
By Lemma \ref{brumer-baker-lem} the map $\widetilde{\rho_\chi}$
is non-zero. 
It follows that the map $\rho_\chi$ can only be zero if the image of $\widetilde{\rho_\chi}$ 
is contained in, and hence coincides with,
the kernel of the projection map 
$\pi_{\chi, \gamma}$, i.e.\@ the
submodule of $e_\chi \Q_p (\chi) \otimes_{\Z_p} \p_v$ generated by $\omega_\chi  (\delta)$.
Thus, the map $\rho_\chi$ is non-zero for all $\chi \neq 1$ if $\delta$ is not a $\Z_p^\times$-multiple of an element in the set $\{ \delta_\chi \mid \chi \in \widehat{\cG}_W \}$, where $\delta_\chi$ denotes a topological generator of 
\[
\omega_\chi^{-1} \Big ( \widetilde{\rho_\chi} ( e_\chi \Q_p (\chi) U_{K, W, T}) \cap \omega_\chi (\gal{F_\infty}{k} ) \Big). 
\]
We now claim that we can choose an integer $N \geq 1$ such that all $\Z_p$-extensions in the set
\[
\Omega (N) = \{ k_\delta \mid \delta = \gamma_\anti^{p^n} \cdot \gamma_\cyc \text{ for some } n \geq N \}
\]
satisfy all of the conditions (a) -- (c). Indeed, if $N$ is big enough such that for $n \geq N$ none of the elements $\gamma_\anti^{p^n} \cdot \gamma_\cyc$ is a $\Z_p^\times$-multiple of an element in $\{ \delta_\chi \mid \chi \in \widehat{\cG}_W \}$, then each $\Z_p$-extension in $\Omega (N)$ will have property (a). Note that $k_{\delta} \cap k^\cyc = k^\cyc_n$ if $\delta = \gamma_\anti^{p^n} \cdot \gamma_\cyc$. Since no finite place splits completely in $k^\cyc / k$, we may therefore choose $N$ such that the second part of (b) is satisfied for each element of $\Omega (N)$. The first part of (b), in turn, follows from a result of Emsalem \cite{emsalem}
which, as a particular case, asserts that in any $\Z_p$-extension of $k$ that is not the anticyclotomic extension at most two finite primes can split completely.
\end{proof}

\section{Abelian extensions of an imaginary quadratic field} \label{iq-section}

In this section we specialise to the case where the base field $k$ is imaginary quadratic.

\subsection{The conjecture of Mazur--Rubin and Sano for elliptic units} \label{set-up-section-iq}

Fix an imaginary quadratic field $k$ and a prime number $p$. We will often distinguish between two cases:
\begin{enumerate}[label=$\bullet$]
    \item \textit{(split case)} The rational prime $p$ splits in $k$. In this case we
    fix a choice of prime ideal $\p \subseteq \bigO_{k}$ above $p$, i.e.\@ we then have $p \bigO_{k}= \frp \overline{\frp}$ with $\frp \neq \overline{\frp}$. 
    \item \textit{(non-split case)} The prime $p$ is either inert in $k$, i.e.\@ $p\bigO_{k}=\frp$ is prime, or ramified, i.e.\@ $p \bigO_{k}=\frp^{2}$.
\end{enumerate}
Fix a finite abelian extension $K / k$ and define $k_\infty$ to be 
\begin{enumerate}[label=$\bullet$]
\item the unique $\Z_p$-extension of $k$ unramified outside $\p$, in the split case,
\item any $\Z_p$-extension of $k$ in which only finitely many finite places split completely, none of them ramified in $K / k$. 
\end{enumerate}
As in \S\,\ref{set-up-section} we then set $K_\infty = K \cdot k_\infty$, write $K_n$ for the $n$-th layer of the $\Z_p$-extension $K_\infty / K$, and use the notations $\cG, \cG_{n}, \Gamma_{n}, \Gamma^n$ and $\bLambda$ etc. We also note that, in the split case, no finite place splits completely in $k_\infty / k$, see \cite[Ch.~II, Prop.~1.9]{deS87}. \smallskip \\
Fix a prime ideal $\a \subsetneq \bigO_k$ that does not split completely in $k_\infty / k$ and is coprime to $6 \p \m$, where $\m = \m_K$ denotes the conductor of $K$. The set $T = \{ \a \}$ then has the property that $U_{E, \Sigma, T}$ is $\Z_p$-torsion free for every subfield $E$ of $K_\infty / k$. \smallskip \\
In the notation of \S\,\ref{set-up-section} we take $V = S_\infty (k)$ and $S$ a finite set which contains $S_\infty (k) \cup S_\ram (K / k)$.
Recall that in \S\,\ref{set-up-section} we have also fixed a proper subset $V' \subsetneq \Sigma$ consisting of places that split completely in $K / k$, and have set $e$ to be the size of $W = V' \setminus V$.  
\smallskip\\
We are now able to state the main result of this subsection.

\begin{thm} \label{MRS-iq-thm}
Let $\kappa_0$ be the Darmon derivative of $\varepsilon_{K_\infty / k, \Sigma, T}$ with respect to some topological generator $\gamma$ of $\Gamma$. Then we have
\[
 \kappa_0 \otimes (\gamma - 1)^e =  (\Rec_W \circ \Ord_W^{-1}) ( \varepsilon^V_{K / k, \Sigma \setminus W, T}).
\]
In particular, Conjecture \ref{MRS} holds for the data $(k_{\infty} / k, K, S, T, V')$ fixed above. 
\end{thm}

\begin{rk}
In the split case
our methods only allow to prove Conjecture~\ref{MRS} for the unique $\Z_p$-extension $k_\infty / k$ which is unramified outside $\p$. However, this is sufficient to establish the relevant case of the equivariant Tamagawa Number Conjecture in this setting (see Theorem~\ref{thm_iq_main_thm}) and this, in turn, implies Conjecture~\ref{MRS} for any choice of $\Z_p$-extension in which no finite place contained in $\Sigma$ splits completely via \cite[Lem.\@ 5.17]{BKS2} (cf.\@ also \cite[Thm.\@ 5.16]{BKS}).
\end{rk}

The proof of Theorem \ref{MRS-iq-thm} occupies the remainder of this subsection and we shall appeal to relative Lubin-Tate theory \cite[Ch.~I]{deS87} during its course. In order to do this, we first need to establish a little more notation. \medskip \\
Let $H$ be a finite extension of $\Q_p$ and denote the cardinality of its residue field $\faktor{\bigO_H}{\p_H}$ by $q$. We fix an integer $d > 0$ and let $H'$ be the unramified extension of $H$ of degree $d$. We write $\varphi \in \gal{H'}{H}$ for the arithmetic Frobenius automorphism. \\
Fix an element $\xi \in H^\times$ such that $\ord_H ( \xi) = d$.  For each power series $f$ satisfying Frobenius-like properties (for details see \cite[Ch. I]{deS87}) there exists a unique one-dimensional commutative formal group law $F_f \in \bigO_{H'} \llbracket X, Y \rrbracket$  satisfying $F_{f}^{\varphi} \circ f = f \circ F_{f}$ called a relative Lubin-Tate group (relative to the extension $H'/H$).
We let $W^n_f$ be the group of \textit{division points of level $n$} of $F_f$ and set $\widetilde{W^n_f} = W^n_f \setminus W^{n - 1}_f$ for every $n \in \N$. Then $H'_n = H' (W^{n + 1}_f)$ is a totally ramified extension of $H'$ of degree $q^n (q - 1)$ and $H'_\infty = \bigcup_{n \in \N} H'_n$ is the maximal totally ramified extension of $H'$ that is abelian over $H$. \medskip \\
Fix $\omega_i \in \widetilde W^i_{\varphi^{- i} (f)}$ such that $(\varphi^{- i} f) ( \omega_i) = \omega_{i - 1}$ and let $u \in \varprojlim_n (H'_n)^\times$ be a norm-coherent sequence. 
There is a unique integer $\nu(u)$ such that $u_{n}\bigO_{H_{n}'} = \mathfrak{p}^{\nu(u)}_{H_{n}'}$ for all $n \geq 0$. By \cite[Ch. I, Thm.~2.2]{deS87}
there is a unique power series $\Col_u \in t^{\nu(u)} \bigO_{H'} \llbracket t \rrbracket ^{\times}$ such that
\[
( \varphi^{- (i + 1)} \Col_u) (\omega_{i + 1}) = u_i 
\]
for all $i \geq 0$. This power series $\Col_u$ is called the \textit{Coleman power series} associated to $u$. \medskip \\
Let $\rho \: \gal{H'_\infty}{H} \to \faktor{\Q}{\Z}$ be a character of finite order. Write $H_\rho = (H'_\infty)^{\ker \rho}$ for the field cut out by $\rho$ and choose $m$ minimal with the property that $H_\rho \subseteq H'_m$. \\
If $u \in \varprojlim_n \bigO_{H'_n}^\times$ is a norm-coherent sequence, then class field theory implies that $\NN_{H'_0 / H} (u_0) = 1$. Hilbert's Theorem 90 therefore ensures the existence of an element $\beta_{\sigma, \rho} \in H^\times_\rho$ satisfying $(\sigma - 1) \cdot \beta_{\sigma, \rho} = \NN_{H'_m / H_\rho} ( u_m)$, where $\sigma$ denotes a generator of $\gal{H_\rho}{H}$. \medskip \\
The following is proved in \cite[Cor.~3.17]{BlHo20}. 

\begin{prop} \label{prop_seiriki_thm}
Using the notation introduced above, assume that $\rho (\sigma) = \frac{1}{[H_\rho : H]} + \Z$.
Then we have 
\[
\frac{\ord_{H_\rho} (\beta_{\sigma, \rho})}{e_{H_\rho / H}} = - \rho ( \rec_H ( \NN_{H' / H} (\Col_u (0)))) 
\quad \text{ in } \faktor{\Q}{\Z},
\]
where we write $e_{H_\rho / H}$ for the ramification degree of the extension $H_\rho / H$ and $\rec_H$ denotes the local reciprocity map $H^\times \to \gal{H'_\infty}{H}$. 
\end{prop}

\textit{Proof of Theorem \ref{MRS-iq-thm}:}
First we observe that by \cite[Prop.~4.4\,(iv)]{BKS2} we may reduce to the case 
$W \subseteq S_\ram (K_{\infty} / K)  = \{ \p \}$. 
Since Conjecture \ref{MRS} is trivial if $W = \emptyset$, we may assume that $W = \{ \p \}$.
In particular, $\p \nmid \m$ and hence, because $V'$ is a proper subset of $\Sigma$, there must be a finite place $\q \in \Sigma$ that is different from $\p$. Observe that $\bigO_k^\times \to (\faktor{\bigO_k}{\q^l} )^\times$ is injective if we choose $l$ big enough. We may therefore take the ideal $\mathfrak{f}$ appearing in Example \ref{Rubin--Stark-examples}\,(c) to be an appropriate power of $\q$. Given this, we have that, for any $n \in \N_0$,
\[
 \varepsilon^V_{K_n / k, \Sigma, T} = \NN_{k ( \mathfrak{f} \m \p^{n + 1})/K_{n}} ( \psi_{\mathfrak{f} \m \p^{n + 1}, \mathfrak{a}})
\]
is the elliptic unit defined in \ref{Rubin--Stark-examples}\,(c). 
\medskip \\
Let $n = (\Gamma : \Gamma_{\p}) $ be the index of the decomposition group at $\p$ inside $\Gamma$, i.e.\@ $n$ is maximal such that $\p$ does not split at all in $K_{\infty} / K_{n}$. We shall now first demonstrate that it suffices to prove Conjecture \ref{MRS} for the field $K_{n}$. \\
By Theorem~\ref{order-of-vanishing-prop}\,(a) the Darmon derivatives $\kappa_0$ and $\kappa_n'$ of $\varepsilon_{K_\infty / k, \Sigma, T}$ with respect to topological generators $\gamma \in \Gamma$ and $\gamma^{p^n} \in \Gamma^n$, respectively, exist. By definition these are the bottom values of norm-coherent sequences $(\kappa_m)_m$ and $(\kappa'_{n + m})_m$ which satisfy
\[
(\gamma - 1) \cdot \kappa_m = \varepsilon^V_{K_m, \Sigma, T}
\quad \text{ and } \quad
(\gamma^{p^n} - 1) \cdot \kappa'_{n + m} = \varepsilon^V_{K_{n + m}, \Sigma, T}
\]
for all integers $m \geq 0$. 
It follows that we have
\[
(\gamma - 1) \cdot \NN_{\Gamma_{n}} \cdot \kappa_{n +m}'  = (\gamma^{p^n} - 1) \cdot \kappa'_{n +m} =  \varepsilon^V_{K_{n + m}, \Sigma, T},
\]
hence, by uniqueness, we must have $(\kappa_{n + m})_m = (\NN_{\Gamma_{n}}  \kappa_{n + m}' )_m $ and it follows that $\kappa_0 = \NN_{\Gamma_{n}}^2 \cdot \kappa'_n = p^n  \NN_{\Gamma_{n}} \kappa'_n$. This implies that
\begin{equation} \label{rechnung-1}
    \kappa_0 \otimes (\gamma - 1) =  p^n  \NN_{\Gamma_{n}} \kappa'_n \otimes (\gamma - 1)
     = (\NN_{\Gamma_{n}}\kappa'_n) \otimes (\gamma^{p^n} - 1)
\end{equation}
inside $U_{K, \Sigma, T} \otimes_{\Z_p} \faktor{I (\Gamma)}{I(\Gamma)^2}$. 
As $\faktor{I (\Gamma_{n})}{I(\Gamma_{n})^2}$ is $\Z_p$-torsion free, the inclusion $U_{K, \Sigma, T} \hookrightarrow U_{K_{n}, \Sigma, T}$ induces an injection
\begin{equation} \label{injection-1}
U_{K, \Sigma, T} \otimes_{\Z_p} \faktor{I (\Gamma_{n})}{I(\Gamma_{n})^2} 
\hookrightarrow 
U_{K_{n}} \otimes_{\Z_p} \faktor{I (\Gamma_{n})}{I(\Gamma_{n})^2}
\end{equation}
that allows us to view (\ref{rechnung-1}) as an equality inside the right hand side of (\ref{injection-1}). 
Assuming the validity of the conjecture for $K_{n}$ (in the form (\ref{Statement1})), we may therefore continue the calculation in (\ref{rechnung-1}) as follows:      
\begin{align} \nonumber 
     \NN_{\Gamma_{n}} (\kappa'_n \otimes (\gamma^{p^n} - 1))
    & = \NN_{\Gamma_{n}} (\Rec_W \circ \Ord_W^{-1}) ( \varepsilon^V_{K_n, \Sigma \setminus W, T})  \\ \nonumber 
    & = (\Rec_W \circ \Ord_W^{-1}) ( \NN_{\Gamma_{n}} \varepsilon^V_{K_n, \Sigma \setminus W, T}) \\ \label{rechnung-2}
    & = (\Rec_W \circ \Ord_W^{-1}) ( \varepsilon^V_{K, \Sigma \setminus W, T} ).
\end{align}
Observe that we have a commutative diagram
\begin{equation*} 
\begin{tikzcd}
       U_{K_{n}, \Sigma, T} \arrow{rr}{\Rec_W \circ \Ord_W^{-1}} & & 
       \Q_p \cdot  U_{K_{n}, \Sigma, T}
       \otimes_{\Z_p} \faktor{I ( \Gamma_{n})}{I (\Gamma_{n})^2}   \\
       U_{K, \Sigma, T} \arrow[hookrightarrow]{u} \arrow{rr}{\Rec_W \circ \Ord_W^{-1}} & &  
       \Q_p \cdot  U_{K, \Sigma, T}
       \otimes_{\Z_p} \faktor{I ( \Gamma_{n})}{I (\Gamma_{n})^2}
       \arrow[hookrightarrow]{u}, 
\end{tikzcd}
\end{equation*}
where the right hand vertical arrow is induced by (\ref{injection-1}).  We caution the reader that the two horizontal arrows, although both labelled $\Rec_W \circ \Ord_W^{-1}$, do not coincide but that inducing the bottom arrow from $\cG$ to $\cG_{n}$ gives the top arrow.\\
Given this commutative diagram, the equations (\ref{rechnung-1}) and (\ref{rechnung-2}) taken together finish the proof of the claim. We therefore may, and will, assume without loss of generality that $\p$ has full decomposition group in $K_{\infty} / K$. \medskip \\
Let $w$ be a place of $K$ above $\p$ and choose an embedding $\iota_w \: \overline{\Q} \hookrightarrow \overline{\Q_p}$ that restricts to $w$ on $K$. In the following, we will denote the completion of a finite abelian extension field $F$ of $k$ at the place induced by $\iota_w$ by $\widetilde{F}$. 
Put $H = \widetilde{K}$ and $H' = \widetilde{k(\ff \m)}$. 
Using that $H'_n = \widetilde{k ( \ff \m \p^{n + 1})}$, 
we can then define a norm-coherent sequence $u = (u_n)_n \in \varprojlim_n \bigO_{H'_n}$ by setting
\[
u_n = \iota_w ( \psi_{\ff \m \p^{n + 1}, \a} ) \quad \text{ for all } n > 0.
\]

\begin{lem} \label{constant-term-lem}
We have
\[ 
\Col_u (0) = \iota_w ( \psi_{\ff \m, \a} ).
\]
\end{lem}

\begin{proof}
In the split case this is \cite[Chp.~II, Sec~4.9, Prop.]{deS87} combined with the evaluation of the power series at zero and an application of the monogeneity relation of Robert's $\psi$-function. A more detailed proof of the split case is given in \cite[Prop.~4.5]{OuVi16} following the same strategy as \cite{deS87}. \medskip \\
We claim that essentially the same proof works in the non-split case. First observe that in the proof of part (i) of the cited Proposition in \cite{deS87} the fact that the prime is split in $k$ is not used. 
In part (ii) the condition that $p$ is split is used to obtain a certain generator of the Tate module $(\omega_{n})$ of the underlying formal group $\widehat{E}$ (because in this case the formal group $\widehat{E}$ is isomorphic to $\widehat{\mathbb{G}}_{m}$ and hence of height one). It is then shown that there exist torsion points $u_{n}$ which can be used to give an explicit description of the elements $\omega_{n}$ at each level \cite[Chp.~II, Sec.~4.4, (12)]{deS87}. In the non-split case one can now invert the strategy: Indeed, it is easy to see that there exist torsion points $u_{n}$ such that the explicit description given in (12) is a generator of the Tate module of $\widehat{E}$. Using this as the definition of $(\omega_{n})$, the remaining steps in the proof are exactly as in the split case. 
\end{proof}

Recall that we have fixed a topological generator $\gamma$ of $\Gamma$ above. 
We define the isomorphisms
\begin{align*}
  s_{\gamma}\:&  \Gamma  \longrightarrow \Z_{p}, &   & s_{\gamma, n} \: \Gamma_{n} \longrightarrow \nZ{p^{n}} \\ 
  & \gamma^{a} \longmapsto a  &   &  \phantom{s_{\gamma, n} \: } \gamma^{a} \longmapsto a \bmod p^{n} \Z 
\end{align*}
and the character
\begin{align}
\label{def_character}
\rho_{\gamma , n} \: \gal{H'_{\infty}}{H} \xrightarrow{\pi_{n}} \gal{\widetilde{K_{n}}}{H} \xrightarrow{s_{\gamma ,n }} \nZ{p^{n}}
\cong \faktor{\tfrac{1}{p^n} \Z}{\Z}
,
\end{align}
where $\pi_{n}$ is the natural projection map induced by restriction. By definition, $\rho_{\gamma, n}$ is a character of finite order with kernel $\gal{H'_\infty}{\widetilde{K_{n}}}$, hence Proposition \ref{prop_seiriki_thm} combines with Lemma~\ref{constant-term-lem} to reveal that 
\begin{align} \label{seiriki-equn-1}
    \frac{\ord_{\widetilde{K_{n}}}(\beta_{\gamma_{n}, \rho_{\gamma, n}})}{e_{\widetilde{K_{n}} / H}} 
    \equiv - \frac{
    s_{\gamma, n} (\pi_{n}(\rec_{H}(\NN_{H'/ H} (\psi_{\ff \m, \a})))}
    {p^n}
    \bmod \Z
\end{align}
for all $n \geq 0$.  
By definition the Darmon derivative $\kappa_0$ of $\varepsilon_{K_\infty / k, \Sigma, T}$ with respect to $\gamma$ is the bottom value of a norm-coherent sequence $\kappa = (\kappa_n) \in \varprojlim_n U_{K_{n}, \Sigma}$ that satisfies
\[
(\gamma - 1) \iota_w ( \kappa_n) = \iota_w (\varepsilon^V_{K_n / k, \Sigma, T}) = \NN_{H'_m / H_{\rho_{\gamma, n}}} ( u_m),
\]
thus we may take $\beta_{\gamma_n, \rho_{\gamma, n}} \equiv \iota_w (\kappa_n) \mod \widetilde{K}^\times$. We now obtain from (\ref{seiriki-equn-1}) that
\begin{align*}
\ord_{H} (  \iota_w (\kappa_0)) & =
\ord_H (
 \NN_{\widetilde{K_{n}} / H} (\iota_w (\kappa_n))) =
\tfrac{p^n}{e_{\widetilde{K_{n}} / H}} \cdot 
\ord_{\widetilde{K_n}} ( \iota_w (\kappa_n)) \\
& \equiv 
 - s_{\gamma, n} (\pi_{n}(\rec_{H}(\NN_{H'/ H} (\psi_{\ff \m, \a}))) \mod p^{n} \Z.
\end{align*}
Taking the limit over $n$ then gives 
\begin{equation} \label{eqn_ord_rec_equal_final}
\ord_w (\kappa_0) = - s_\gamma (\rec_w ( \varepsilon^V_{K, \Sigma \setminus W, T} ))
\end{equation}
as an equality in $\Z_p$. By repeating the argument we also obtain equation (\ref{eqn_ord_rec_equal_final}) for the places $\sigma w$, where $\sigma \in \cG$. Collating these equations, we find that
\begin{align*}
    \Ord_W (\kappa_0) \otimes (\gamma - 1) & = \sum_{\sigma \in \cG_\chi} \ord_{\sigma w} (\kappa_0)  \sigma \otimes (\gamma - 1) \\
    & = - \sum_{\sigma \in \cG_\chi} s_\gamma (\rec_{\sigma w} (\varepsilon^V_{K, \Sigma \setminus W, T} )) \sigma \otimes (\gamma - 1) \\
    & = - \Rec_W ( \varepsilon^V_{K, \Sigma \setminus W, T} ).
\end{align*}
By Lemma \ref{MRS-Solomon-formulation} this concludes the proof of Theorem \ref{MRS-iq-thm}. 
\qed

\subsection{The equivariant Iwasawa Main Conjecture}

In this section we prove a suitable variant of the equivariant Iwasawa Main Conjecture for abelian extensions of an imaginary quadratic field.
In this setting, numerous results on the Iwasawa Main Conjecture have already appeared in the literature, both in classical and equivariant  formulations  (cf.\@ \cite{Rub88}, \cite{Rubin91}, \cite{Rub94}, \cite{Ble06}, \cite{Fla09}, \cite{JoKi11}, \cite{Vig13}). 
However, we require a result that is both slightly more general and of a different shape than is available in the literature thus far. \medskip \\ 
 Suppose to be given an abelian extension $L_\infty / k$ such that $\gal{L_\infty}{k} \cong \Gamma \times \Delta$, where $\Delta$ is a finite abelian group and $\Gamma \cong \Z_p^d$ for an integer $d > 0$.
Note that $d \in \{1, 2 \}$ as a consequence of the known validity of Leopoldt's Conjecture for the imaginary quadratic field $k$. 
 \\
We also fix a finite set $\Sigma$ of places of $k$ that contains $S_\infty (k) \cup S_p ( k)$ and a finite set $T$ of places of $k$ that is disjoint from $\Sigma$. Assume that no finite place contained in $\Sigma$ splits completely in the $\Z_p^d$-extension $L_\infty^\Delta / k$. \medskip \\
As before we write $\bLambda = \Z_p[\Delta] \llbracket \Gamma \rrbracket$ for the relevant equivariant Iwasawa algebra and denote its total field of fractions by $\cQ (\bLambda)$. One can then define a perfect complex $D^\bullet_{L_\infty, \Sigma, T}$ as in (\ref{limit-complex})
and define a map
\begin{align}
\Det_{\bLambda} (D^\bullet_{L_\infty, \Sigma, T}) & \hookrightarrow \cQ ( \bLambda) \otimes_\bLambda \Det_{\bLambda} (D^\bullet_{L_\infty, \Sigma, T}) 
\nonumber \\
& \cong \Det_{\cQ (\bLambda)} ( \cQ ( \bLambda) \otimes^\mathbb{L}_\bLambda D^\bullet_{L_\infty, \Sigma, T}) 
\nonumber \\
& \cong \big (\cQ (\bLambda) \otimes_\bLambda U_{L_\infty, \Sigma, T} \big)
\otimes_{\cQ ( \bLambda)} \big( \cQ ( \bLambda) \otimes_{\bLambda} Y_{L_\infty, S_\infty (k)} \big)^\ast \nonumber \\
& \cong \cQ (\bLambda) \otimes_\bLambda U_{L_\infty, \Sigma, T},
\label{projection-map-infinite}
\end{align}
where the first isomorphism follows from a well-known property of the determinant functor, the second isomorphism is the natural `passage-to-cohomology' map, and the last isomorphism is due to the isomorphism $Y_{L_\infty, S_\infty (k)} \cong \bLambda$ obtained from our fixed choice of extension of the unique infinite place of $k$ to $L_\infty$.\\ 
The map (\ref{projection-map-infinite}) then restricts to a map
\[
\Theta_{L_\infty / k, \Sigma, T}^1 \:  
    \Det_{\bLambda} (D^\bullet_{L_\infty, \Sigma, T}) 
    \hookrightarrow U_{L_\infty, \Sigma, T}^{\ast \ast} \cong U_{L_\infty, \Sigma, T},
    \]
see \cite[Lem.\@ 3.12]{BullachDaoud} for more details.    
\smallskip\\
We now recall the (higher-rank) equivariant Iwasawa Main Conjecture in this setting as proposed in \cite[Conj.\@ 3.1 and Rk.\@ 3.3]{BKS2}.

\begin{conj} \label{eIMC}
There exists a $\bLambda$-basis $\mathcal{L}_{L_\infty / k, \Sigma, T}$ of $\Det_\bLambda (D^\bullet_{L_\infty, \Sigma, T})$ such that
\[
\Theta^1_{L_\infty / k, \Sigma, T} ( \mathcal{L}_{L_\infty / k, \Sigma, T}) = \varepsilon_{L_\infty / k, \Sigma, T}.
\]
\end{conj}

Fix a prime ideal $\p$ of $k$ above $p$ as in \S\,\ref{set-up-section-iq}. The main result of this is subsection is as follows.  

\begin{thm} \label{appendix-main-result}
Let $K / k$ be an abelian extension and put $L_\infty = K l_\infty$, where $l_\infty$ is the maximal $\Z_p$-power extension of $k$ unramified outside $\p$. Assume the following condition:
\begin{itemize}
    \item[($\ast$)] $\gal{L_\infty}{k}$ is $p$-torsion free or the Iwasawa $\mu$-invariant of $A_{\Sigma} (L_\infty)$ (as a $\Z_p \llbracket \Gamma \rrbracket$-module) vanishes.
\end{itemize}
Then Conjecture \ref{eIMC} holds for $(L_\infty / k, \Sigma, T)$. In particular, Conjecture \ref{eIMC} holds for $(K_\infty / k, \Sigma, T)$ with $K_\infty = K k_\infty$ if one takes $k_\infty / k$ to be any of the $\Z_p$-extensions described at the beginning of \S\,\ref{set-up-section-iq}. 
\end{thm}

\begin{proof}
Let us first prove that it is indeed enough to prove Conjecture \ref{eIMC} for $L_\infty / k$. If $p$ is split in $k / \Q$, then $l_\infty$ and $k_\infty$ agree and so the claim is clear in this case. In the non-split case, $l_\infty$ is the maximal $\Z_p$-power extension of $k$ and hence $K_\infty$ is contained in $L_\infty$. We then have a commutative diagram  
\begin{equation} \label{descent-diagram-1}
\begin{tikzcd}
\Det_{\bLambda'} ( D^\bullet_{L_\infty, \Sigma, T} ) \arrow{rr}{\Theta^1_{L_\infty / k, \Sigma, T}} \arrow[twoheadrightarrow]{d}[left]{\varpi_{L_\infty / K_\infty}} & &
U_{L_\infty, \Sigma, T} \arrow{d}{\NN_{L_\infty / K_\infty}} \\
\Det_{\bLambda} ( D^\bullet_{K_\infty, \Sigma, T} )
\arrow{rr}{\Theta^1_{K_\infty / k, \Sigma, T}} & &
U_{K_\infty, \Sigma, T},
\end{tikzcd}
\end{equation}
where the left hand vertical map $\varpi_{L_\infty / K_\infty}$ is induced by the isomorphism (cf.\@ Proposition \ref{propeties-iwasawa-complex}\,(c)\,(ii))
\[
D^\bullet_{L_\infty, \Sigma, T} \otimes^\mathbb{L}_{\Z_p \llbracket \gal{L_\infty}{k} \rrbracket} \Z_p \llbracket \gal{K_\infty}{k} \rrbracket
\cong D^\bullet_{K_\infty, \Sigma, T}.
\]
The claim now follows directly from the above commutative diagram (\ref{descent-diagram-1}). \smallskip \\
To prove Conjecture \ref{eIMC} for $L_\infty / k$, we first note that the explicit condition ($\ast$) ensures that $A_\Sigma (L_\infty)$ has projective dimension at most one after localising at any height-one prime $\p$ of $\bLambda$. By \cite[Lem.\@ 6.2\,(b)]{scarcity} it is therefore enough to show that one has an inclusion
\[
\im (\varepsilon_{L_\infty / k, \Sigma, T})^{\ast \ast} \subseteq \Fitt^0_\bLambda ( A_{\Sigma, T} (L_\infty))^{\ast \ast} \cdot \Fitt^0_\bLambda ( X_{L_\infty, \Sigma \setminus S_\infty (k)})^{\ast \ast}
\]
and this can be done using the theory of Euler systems (see the proof of \cite[Thm.\@ 6.5\,(b)]{scarcity}, where the above inclusion, which agrees with (31) of \textit{loc.\@ cit.}, is verified). 
\end{proof}

To end this subsection we clarify the nature of condition ($\ast$). 

\begin{prop} \label{mu}
Let $K / k$ be an abelian extension and put $K_\infty = K \cdot l_\infty$, where $l_\infty$ is the maximal $\Z_p$-power extension of $k$ unramified outside $\p$. The $\mu$-invariant of $A_\Sigma (L_\infty)$ (as a $\Z_p \llbracket \Gamma \rrbracket$-module) vanishes in each of the following cases:
\begin{liste}
\item The prime number $p$ splits in $k / \Q$,
\item the degree $[K : k]$ is a power of $p$,
\item there is a $\Z_p$-extension $F_\infty$ of $K$ contained in $L_\infty$ in which no prime above $p$ splits completely and which is such that the $\mu$-invariant of $A (F_\infty)$ (as a $\Z_p \llbracket \gal{F_\infty}{K} \rrbracket$-module) vanishes,
\item $A_{S_p} (K)$ vanishes and $|S_p (K) | = 1$.
\end{liste}
\end{prop}

\begin{rk}
Iwasawa has conjectured that statement (c) in Proposition \ref{mu} is always satisfied if one takes $F_\infty$ to be the cyclotomic $\Z_p$-extension of $K$.
\end{rk}

\begin{proof}
It is well-known that the Iwasawa $\mu$-invariants of $A_\Sigma (L_\infty)$ and $A (L_\infty)$ agree because no finite prime splits completely in $L_\infty$ (see \cite[Ch.\@ II, \S\,1.9, Prop.\@]{deS87}), hence it suffices to discuss the vanishing of the latter. \\
In the situation of (a) the required vanishing follows from the main results of \cite{Gil85} (for $p > 3$) and \cite{OuVi16} (for $p \in \{2, 3\}$).\\
Let $k_\infty^\mathrm{cyc}$ and $K^\cyc_\infty$ be the cyclotomic $\Z_p$-extensions of $k$ and $K$, respectively. From \cite[Thm.~2]{Iwa73_b}
we know that the vanishing of the $\mu$-invariant of $A ( K^\mathrm{cyc}_\infty)$
is implied by the vanishing of the $\mu$-invariant of $A ( K^P \cdot k^\mathrm{cyc}_\infty)$, where $K^P$ denotes the fixed field of the $p$-Sylow subgroup $P$ of $\cG = \gal{K}{k}$. This combines with the Theorem of Ferrero-Washington \cite{ferrero-washington} to imply the claim for (b) once we have verified that it is valid if (c) holds. \\
To do this, we may assume that $p$ is not split in $k / \Q$ because we have already dealt with split primes in (a). In this case, $K_\infty / K$ is a $\Z_p^2$-extension in which all primes above $p$ are finitely decomposed. Given this, statement (c) implies the claim by \cite[Prop.~4.1 and Cor.~4.8]{Cuo80}.\\
Finally, as is well-known, (d) follows from Nakayama's Lemma using that $(A_{S_p} (K_\infty))_\Gamma \cong A_{S_p} (K)$ if $|S_p (K) | = 1$. 
\end{proof}

\subsection{Proof of Theorem~\ref{thm-B}}

We are finally in a position to prove Theorem \ref{thm-B} from the introduction. 

\begin{thm} \label{thm_iq_main_thm}
Let $p$ be a prime number, $k$ an imaginary quadratic field, and $K / k$ a finite abelian Galois extension with Galois group $\cG$.
\begin{liste}
\item If $p$ splits in $k$, then $\mathrm{eTNC} ( h^0 (\Spec (K)), \Z_p [\cG])$ holds. 
\item If $p$ does not split in $k$, then $\mathrm{eTNC} ( h^0 (\Spec (K)), \Z_p [\cG])$ holds if the following condition is satisfied:
\begin{itemize}
    \item [($\ast$)] Let $l_\infty$ be the maximal $\Z_p$-power extension of $k$ and put $L_\infty = K \cdot l_\infty$. The Iwasawa $\mu$-invariant of $A_\Sigma (L_\infty)$ (as a $\Z_p \llbracket \gal{L_\infty}{K} \rrbracket$-module) vanishes or $\gal{L_\infty}{k}$ is $p$-torsion free.
\end{itemize}
\end{liste}
\end{thm}

\begin{proof}
This will follow from the equivariant Iwasawa Main Conjecture proved in Theorem \ref{appendix-main-result} and the descent argument of Burns, Kurihara and Sano in \cite[Thm.~5.2]{BKS2}.\smallskip \\
To do this, we first need to introduce some notation. In the split case, we take $k_\infty$ to be the unique $\Z_p$-extension of $k$ that is unramified outside $\p$. In the non-split case, we take $k_\infty$ to be one of the $\Z_p$-extensions of $k$ provided by Theorem~\ref{exists-Zp-extension}.\\
For any character $\chi \in \widehat{\cG}$ we moreover introduce the following notation:
\begin{enumerate}[label=$\bullet$]
    \item $K_\chi = K^{\ker \chi}$ the field cut out by the character $\chi$, and $\cG_\chi = \gal{K_\chi}{k}$ is Galois group,
        \item $K_{\chi, \infty} = K_\chi \cdot k_\infty$ the composite of $K_\chi$ with 
        $k_\infty$, and $\Gamma_\chi = \gal{K_{\chi, \infty}}{K_\chi}$.
    \end{enumerate}
In addition, we define
\[
V'_\chi = \begin{cases}
S_\spc (K_\chi / k) \cap \Sigma & \text{ if } \chi \neq 1, \\
\Sigma \setminus \{ \p \} & \text{ if } \chi = 1,
\end{cases}
\]
and set $W_\chi = V'_\chi \setminus V$, where $V = S_\infty (k)$. By enlarging $S$ if necessary we may assume that $S_p (k) \subseteq \Sigma$ and that $V'_\chi$ is a proper subset of $\Sigma$ for all $\chi \in \widehat{\cG}$.
\smallskip \\
Let us now address each condition required to apply the general result \cite[Thm.~5.2]{BKS2} separately:
\begin{enumerate}[label=$\bullet$]
\item The equivariant Iwasawa Main Conjecture holds for $(K_\infty / k, \Sigma, T)$ by Theorem \ref{appendix-main-result}. 
\item The Iwasawa-theoretic Mazur--Rubin Sano Conjecture (in the formulation \cite[Conj.\@ 4.2]{BKS2}) holds for the data 
$(K_{\chi, \infty} / k, K_\chi, S, T, V'_\chi)$: If $\chi$ is non-trivial, then this is proved in Theorem \ref{MRS-iq-thm}.
For the trivial character the set $V'_\chi = \Sigma \setminus \{ \p \}$ consists only of places unramified in $k_\infty / k$, hence in this case the conjecture holds as a consequence of \cite[Prop.~4.4\,(iv)]{BKS2}.
\item Condition (F) (as stated in \ref{condition-F}) for $K_\infty / K$ is valid: In the split case this is Remark \ref{new-gross-kuzmin-thm-rk}\,(b), in the non-split case this is Theorem \ref{exists-Zp-extension}).
\end{enumerate}
 This concludes the proof of Theorem \ref{thm_iq_main_thm}.
\end{proof}

From Theorem \ref{thm_iq_main_thm} and Proposition \ref{mu}\,(b) we immediately obtain the following result towards the integral eTNC.

\begin{cor}
\label{cor_finsterau}
If all prime factors of $[K:k]$ are split in $k$ or $[K : k]$ is a prime power, then  $\mathrm{eTNC} ( h^0 (\Spec (K)), \Z[\cG])$ holds.
\end{cor}

\begin{remark} \label{rk_finsterau}
\begin{liste}
    \item If $p \nmid h_k [K : k]$, where $h_k$ denotes the class number of $k$, then the validity of $\mathrm{eTNC} ( h^0 (\Spec (K)), \Z_p [\cG])$ also follows from unpublished work of Bley \cite[Part~II, Thm.~1.1]{Ble98} on the Strong Stark Conjecture. It should be straightforward to strengthen said result to cover all primes $p \nmid [K : k]$ by taking into account the improvements of \cite[\S\,3]{Rubin91} provided in \cite{Rub94}. We remark that even this expected strengthening is covered by Theorem \ref{thm_iq_main_thm}.  
    \item As illustrated by
Corollary \ref{cor_finsterau},
the validity of the $p$-part of the eTNC for split primes $p \mid 2 h_k$ allows for a significant improvement towards the integral eTNC. Previously,
one had to restrict to cases where $k$ is one of only nine imaginary quadratic fields of class number one and all prime factors of $[K : k]$ are split in $k$ to obtain unconditional results towards the validity of the
eTNC for the pair $(h^0 (\Spec (K)), \Z[\cG])$. 
\end{liste}

\end{remark}

\renewcommand{\emph}[1]{\textit{#1}}

\addcontentsline{toc}{section}{References}

\printbibliography
\enlargethispage{2cm}

\small
\enlargethispage{2cm}
\textsc{King's College London,
Department of Mathematics,
London WC2R 2LS,
UK} \\
\textit{Email address:} \href{mailto:dominik.bullach@kcl.ac.uk}{dominik.bullach@kcl.ac.uk}\\

\textsc{Universit\"at der Bundeswehr M\"unchen,
Kompetenzzentrum Krisenfr\"uherkennung \\
Werner-Heisenberg-Weg 39, 85577 Neubiberg,
Germany}\\
\textit{Email address:} \href{mailto:martin.hofer@unibw.de}{martin.hofer@unibw.de}

\end{document}